\documentclass{amsart}
\usepackage{amssymb,mathtools}
\usepackage{stmaryrd}
\usepackage[british]{babel}
\usepackage{enumitem}
\usepackage[latin1]{inputenc}
\usepackage{url}
\usepackage{hyperref}
\usepackage{turnstile}

\usepackage{bussproofs}

\newtheorem{theorem}{Theorem}
\numberwithin{theorem}{section}
\newtheorem{corollary}[theorem]{Corollary}
\newtheorem{lemma}[theorem]{Lemma}
\newtheorem{proposition}[theorem]{Proposition}
\theoremstyle{definition}
\newtheorem{definition}[theorem]{Definition}
\newtheorem{remark}[theorem]{Remark}

\newtheorem{exercise}[theorem]{Exercise}

\newcommand{\gap}{\mathrel{\trianglelefteq}}
\newcommand{\tl}{\mathrel{\vartriangleleft}}
\newcommand{\supp}{\operatorname{supp}}
\newcommand{\rk}{\operatorname{rk}}
\newcommand{\cl}{\operatorname{Cl}}
\newcommand{\lf}{\operatorname{LF}}

\newcommand{\pica}{\Pi^1_1\text{-}\mathsf{CA}_0}
\newcommand{\aca}{\mathsf{ACA}}
\newcommand{\rca}{\mathsf{RCA}}
\newcommand{\id}{\mathsf{ID}}
\newcommand{\bho}{\vartheta(\varepsilon_{\Omega+1})}
\newcommand{\lpax}{\mathcal L_{\mathsf{PA}}^X}
\newcommand{\lpa}{\mathcal L_{\mathsf{PA}}}
\newcommand{\lid}{\mathcal L_{\mathsf{ID}}}
\newcommand{\lido}{\mathcal L_{\mathsf{ID}}^\Omega}
\newcommand{\bh}{\vartheta(\varepsilon_{\Omega+1})}
\newcommand{\rng}{\operatorname{rng}}

\title[A Second Course on Ordinal Analysis]{Impredicativity and Trees with Gap Condition:\\ A Second Course on Ordinal Analysis}
\author{Anton Freund}

\address{Anton Freund, Department of Mathematics, Technical University of Darmstadt, Schloss\-garten\-str.~7, 64289~Darmstadt, Germany}
\email{freund@mathematik.tu-darmstadt.de}

\begin{document}

\begin{abstract}
These lecture notes introduce central notions of impredicative ordinal analysis, such as the Bachmann-Howard ordinal and the method of collapsing, which transforms uncountable proof trees into countable ones. Specifically, we analyze parameter-free $\Pi^1_1$-comprehension and show that it cannot prove the extended Kruskal theorem due to Harvey Friedman (not even for two labels). In terms of prerequisites, we build on a previous lecture on the ordinal analysis of Peano arithmetic. The present material is intended for 12~lectures and 6~exercise sessions of 90~minutes each. 
\end{abstract}


\maketitle

\section{Introduction}

Ordinal analysis measures the strength of mathematical theorems and axiom systems by ordinal numbers, which can be represented by computable well orders. It allows to prove metamathematical results about independence, conservativity and the complexity of provably total algorithms. In a previous lecture course~\cite{first-course}, we have presented the classical ordinal analysis of Peano arithmetic, which is originally due to Gerhard Gentzen~\cite{gentzen36,gentzen43}. We have derived that conservative extensions of Peano arithmetic cannot prove Kruskal's theorem for binary trees. This yields a mathematical example for the incompleteness phenomenon from G\"odel's theorems. Related but much stronger examples were given by Harvey Friedman (see~\cite{simpson85}).

In the present lecture, we discuss the ordinal analysis of an axiom system that is considerably stronger than Peano arithmetic: the theory $\pica^-$ of \mbox{parameter-free} $\Pi^1_1$-comprehension. To explain the latter, we recall the language $\mathcal L_2$ of second order arithmetic, as presented in Section~6 of the previous lecture~\cite{first-course} (see also the standard textbook by Stephen Simpson~\cite{simpson09}). The comprehension axiom for an $\mathcal L_2$-formula $\varphi$ and a designated variable~$x$ is the universal closure of the formula
\begin{equation}\tag{$\varphi\text{-}\mathsf{CA}$}
\exists X\subseteq\mathbb N\forall x\in\mathbb N\,(x\in X\leftrightarrow\varphi(x,y_1,\ldots,y_m,Y_1,\ldots,Y_n)).
\end{equation}
In this context, we will say that $\varphi$ is parameter-free if it does not contain free set variables, i.\,e., if we have $n=0$. We do allow number variables $y_1,\ldots,y_m$, noting that they could be coded into~$x$. As in the previous course~\cite{first-course}, a formula $\theta$ is called arithmetical if it does not contain bound set variables. A~$\Pi^1_1$-formula has the form $\forall Z\subseteq\mathbb N.\,\theta$ for arithmetical~$\theta$. In the previous course~\cite{first-course} we have encountered the theory $\aca_0$ of arithmetical comprehension. To obtain $\pica^-$ from $\aca_0$, we add the axiom ($\varphi\text{-}\mathsf{CA}$) for each parameter-free $\Pi^1_1$-formula~$\varphi$. Let us note that ($\varphi\text{-}\mathsf{CA}$) is still an axiom for arithmetical $\varphi$ with parameters. As before, the subscript~$0$ indicates that induction is only available for properties that are given by a set variable, i.\,e., in the form
\begin{equation*}
\forall X\subseteq\mathbb N\,\big(0\in X\land\forall x\in\mathbb N\,(x\in X\to x+1\in X)\to\forall x\in\mathbb N.\,x\in X\big).
\end{equation*}
The supercript in $\pica^-$ marks the restriction on parameters. Accordingly, we write $\pica$ for the theory that results from~$\aca_0$ (equivalently from~$\rca_0$) when we add the axioms ($\varphi\text{-}\mathsf{CA}$) for all $\Pi^1_1$-formulas~$\varphi$.

The theory $\pica$ is the strongest of the `big five' axiom systems that are central to the research programme of reverse mathematics. It captures the precise strength of important mathematical results such as the Cantor-Bendixson theorem or the fact that any countable Abelian group is a direct sum of a divisible and a reduced group (see~\cite{simpson09}). In Remark~\ref{rmk:unprovability} below, we will see that $\pica^-$ does easily prove Kruskal's theorem. Together with the result from the previous lecture~\cite{first-course}, this confirms that $\pica^-$ is (much) stronger than~$\aca_0$. In the next section we will present an extended Kruskal theorem due to Harvey Friedman, which strengthens the original result by a so-called gap condition. This extended Kruskal theorem is unprovable in~$\pica$ (see again~\cite{simpson85}). The gap condition is particularly relevant due to its close connection with the graph minor theorem of Neil Robertson and Paul Seymour~\cite{robertson-seymour-gm}, which has been described as one of the ``deepest theorems that mathematics has to offer" (in the textbook by Reinhard Diestel~\cite{diestel-graph-theory}). As shown by Friedman, Robertson and Seymour~\cite{friedman-robertson-seymour}, the extended Kruskal theorem is equivalent (over a weak base theory) to the graph minor theorem for bounded tree-width. It follows that the (bounded) graph minor theorem is unprovable in $\pica$. This is one of the most spectacular manifestations of G\"odel's theorems in mathematics. It is particularly impressive because the graph minor theorem does not assert the existence of any infinite sets. In more technical terms, it is a $\Pi^1_1$-statement and hence valid in all $\omega$-models. As a consequence, its independence cannot be established by typical methods of computability or set theory (cf.~the discussion of Specker sequences in the introduction to the previous course~\cite{first-course}). It appears that the only known proof relies on the ordinal analysis of~$\pica$. In the present course, we do not prove the independence of the graph minor theorem, but we will see many important ingredients: To facilitate matters, we will give an ordinal analysis of $\pica^-$ rather than~$\pica$. As an application, we show that a fragment of the extended Kruskal theorem is unprovable in $\pica^-$ (see Corollary~\ref{cor:EKT-unprovable} below).

It is a major step from $\aca_0$ to~$\pica^-$. To explain this, we recall that models for $\mathcal L_2$ have the form $(\mathcal N,\mathcal S)$, where $\mathcal S$ is a subset of the powerset of~$\mathcal N$. Let us focus on the case where $\mathcal N$ is the standard structure $\mathbb N$ of natural numbers. The interpretation of an $\mathcal L_2$-formula $\varphi(x)$ is given by
\begin{equation*}
\llbracket\varphi\rrbracket_{\mathcal S}:=\{x\in\mathbb N\,|\,(\mathbb N,\mathcal S)\vDash\varphi(x)\}.
\end{equation*}
With respect to the comprehension principle, we now observe
\begin{equation*}
(\mathbb N,\mathcal S)\vDash(\varphi\text{-}\textsf{CA})\quad\Leftrightarrow\quad\llbracket\varphi\rrbracket_{\mathcal S}\in\mathcal S.
\end{equation*}
The right side reveals a certain circularity, which becomes relevant when we aim to construct models of $\pica^-$ or related theories: To ensure $\llbracket\varphi\rrbracket_{\mathcal S}\in\mathcal S$, we may need to enlarge the set $\mathcal S$. When we do so, however, the meaning of $\llbracket\varphi\rrbracket_{\mathcal S}$ can change. In the special case where $\varphi$ is arithmetical, the interpretation $\llbracket\varphi\rrbracket=\llbracket\varphi\rrbracket_{\mathcal S}$ is independent of~$\mathcal S$. To obtain a model $(\mathbb N,\mathcal S)$ of $\aca_0$, it is enough to set
\begin{equation*}
\mathcal S:=\{\llbracket\varphi\rrbracket\,|\,\varphi(x)\text{ is an arithmetical $\mathcal L_2$-formula}\},
\end{equation*}
as in the proof of Proposition~6.5 from the previous lecture~\cite{first-course}. We point out that the resulting set~$\mathcal S$ is countable. In contrast, the only obvious model of~$\pica^-$ seems to involve the full powerset of~$\mathbb N$, even though a countable model can be obtained via the L\"owenheim-Skolem theorem. A more explicit construction of a countable model can be found in the proof of Theorem~\ref{thm:pica-id-conserv} below.

The previous paragraph relates to the notion of `predicativity', which is an important theme in the foundations of mathematics. Somewhat simplified, a definition of some object~$X$ is impredicative if it involves a condition that ranges over a `large' collection~$\mathcal S$ of which~$X$ is a member. This is the case in the above definition of $X=\llbracket\varphi\rrbracket_{\mathcal S}$, when we take $\mathcal S$ to be the full powerset. A predicative definition is typically associated with a construction `from below', while impredicative definitions are often justified with the help of a given object that is large (uncountable or at least noncomputable). The predicative stance goes back to Hermann Weyl~\cite{weyl-continuum} and has been pursued, in particular, by Solomon Feferman (see, e.\,g.,~\cite{feferman-predicativity} as well as the study by Laura Crosilla~\cite{crosilla-feferman}). While there are different ways to make the notion of predicativity precise (cf.~the criticism of Nik Weaver~\cite{weaver}), the formal analysis due to Feferman~\cite{feferman64} and, independently, Kurt Sch\"utte~\cite{schuette64} is by far the most~influential.

In ordinal analysis, one associates different methods with predicative and with impredicative theories, even though there is fruitful interaction. Predicative theories are often analyzed with the help of cut elimination. The method that we have seen in the first part of this lecture~\cite{first-course} can be generalized considerably, e.\,g., to infinite cut ranks. We will present such a generalization in Section~\ref{sect:Veblen-hierarchy} below, but the focus of the present course is on an impredicative method known as collapsing. Roughly speaking, the idea is to collapse uncountable proof trees into countable ones. This idea admits vast generalizations, which support the strongest known ordinal analyses, due to Toshiyasu Arai~\cite{arai-reflection} and Michael Rathjen~\cite{rathjen-reflection,rathjen-stability,rathjen-Pi12}.

To give a more technical summary for the expert, we will approach $\pica^-$ via the theory~$\id_1$ of non-iterated inductive definitions. As in the textbook by Wolfram Pohlers~\cite{pohlers-proof-theory}, the latter will be analyzed via the method of operator controlled derivations, which is due to Wilfried Buchholz~\cite{buchholz-local-predicativity}. The proof theoretic ordinal of $\pica^-$ and~$\id_1$ is the so-called Bachmann-Howard ordinal~$\vartheta(\varepsilon_{\Omega+1})$. We will approach this ordinal via the gap condition of Harvey Friedman, i.\,e., in combinatorial rather than set theoretic terms.

Some readers may feel that Sections~\ref{sect:iterated-Kruskal} and~\ref{sect:Bachmann-Howard} are overly technical. If one is~prepared to accept Corollary~\ref{cor:bachmann-howard-wo} without proof, an alternative route is available: First read the more conceptual and accessible Section~\ref{sect:ind-def}. Then read Section~\ref{sect:Bachmann-Howard} up to and including Exercise~\ref{ex:Bachmann-Howard}. Skip the material between Definition~\ref{def:emb-T2} and Theorem~\ref{thm:emb-T2} (as well as the entire Section~\ref{sect:iterated-Kruskal}). Read the rest of Section~\ref{sect:Bachmann-Howard} from Corollary~\ref{cor:bachmann-howard-wo}~on. Then read Sections~\ref{sect:operator-control} to~\ref{sect:Veblen-hierarchy} (but ignore Corollary~\ref{cor:EKT-unprovable} or consider the omitted material at this point).

In order to present some topics in detail, we need to omit others entirely, which means that we cannot do justice to the subject as a whole. For a more complete picture of ordinal analysis, we strongly recommend to consider the survey articles by Michael Rathjen~\cite{rathjen-realm,rathjen-icm} and by Rathjen and Wilfried Sieg~\cite{rathjen-sieg-stanford}. Classical textbooks are due to Kurt Sch\"utte~\cite{schuette77}, Gaisi Takeuti~\cite{takeuti-proof-theory}, Jean-Yves Girard~\cite{girard87} and Wolfram Pohlers~\cite{pohlers-proof-theory}, as well as to Helmut Schwichtenberg and Stanley Wainer~\cite{schwichtenberg-wainer}.

\section{Iterating Kruskal's theorem: Friedman's gap condition}\label{sect:iterated-Kruskal}

In this section, we show that iterated applications of Kruskal's theorem lead to partial orders with a certain `gap condition', which is due to Harvey Friedman (see the presentation by Stephen Simpson~\cite{simpson85}). This results in an extended Kruskal theorem that is unprovable by $\Pi^1_1$-comprehension. Our presentation of the gap~condition builds on~\cite{freund-kruskal-gap}.

A finite multiset consists of a finite set~$a$ and a function $f:a\to\mathbb N\backslash\{0\}$, which determines the multiplicity of its elements. Let $M(X)$ be the set of finite~multi\-sets with elements $a\subseteq X$. One can also characterize $M(X)$ as the quotient of the finite sequences in~$X$ modulo re-ordering. We write $[x_0,\ldots,x_{k-1}]$ for the finite multi\-set $f:\{x_0,\ldots,x_{k-1}\}\to\mathbb N\backslash\{0\}$ where $f(x_i)$ is the cardinality of~\mbox{$\{j<k\,|\,x_i=x_j\}$}. The following definitions may at first seem ad hoc. Below, however, we argue~that they arise naturally when one attempts to iterate Kruskal's theorem.

\begin{definition}
For $N\in\mathbb N$ and any partial order~$X$, we declare that $T_N(X)$ is generated by the following recursive clauses:
\begin{enumerate}[label=(\roman*)]
\item we have an element~$\overline x\in T_N(X)$ for each~$x\in X$,
\item given a multiset $\sigma=[t_0,\ldots,t_{k-1}]$ of elements $t_i\in T_N(X)$ that have already been constructed, we add an element $n\star\sigma$ for each~$n<N$.
\end{enumerate}
Put $r(\overline x):=-1$ and $r(n\star\sigma):=n$. We then define $T^0_N(X):=\{t\in T_N(X)\,|\,r(t)\leq 0\}$.
\end{definition}

Intuitively, the elements of $T_N(X)$ are rooted trees with vertex labels $n<N$ and $x\in X$, where the latter may only occur at leaves. Specifically, $n\star[t_0,\ldots,t_{k-1}]$ corresponds to the tree with immediate subtrees~$t_0,\ldots,t_{k-1}$ and label~$n$ at the root (which is a leaf when we have~$k=0$). As $[t_0,\ldots,t_{k-1}]$ is a multiset (rather than a sequence), we are concerned with unordered trees. Given a binary relation~$\leq$ on~$Y$, we define $\leq^M$ as the binary relation on~$M(Y)$ such that
\begin{equation*}
[x_0,\ldots,x_{k-1}]\leq^M[y_0,\ldots,y_{m-1}]
\end{equation*}
holds precisely when there is an injection $f:\{0,\ldots,k-1\}\to\{0,\ldots,m-1\}$ with $x_i\leq y_{f(i)}$ for all~$i<k$. We point out that there is a close connection with Higman's lemma. In clause~(ii) of the following definition, one can evalute $\sigma\gap^M\tau$ recursively, even when $\gap$ is not defined on all of~$T_N(X)$ yet.

\begin{definition}\label{def:gap-order}
To determine a binary relation $\gap$ on~$T_N(X)$ by recursion, we declare that $s\gap t$ holds precisely when one of the following conditions is satisfied:
\begin{enumerate}[label=(\roman*)]
\item $s=\overline x$ and $t=\overline y$ with $x\leq_X y$,
\item $s=m\star\sigma$ and $t=n\star\tau$ with $m=n$ and $\sigma\gap^M \tau$,
\item $t=n\star[t_0,\ldots,t_{k-1}]$ with $r(s)\leq n$ and $s\gap t_i$ for some~$i<k$.
\end{enumerate}
\end{definition}

Intuitively, we have $s\gap t$ when there is an embedding between the corresponding trees, which must respect the labels in a certain sense. Clause~(ii) corresponds to an embedding that sends the root to the root and the immediate subtrees of~$s$ into different subtrees of~$t$. In the case of clause~(iii), the entire tree~$s$ is mapped into a proper subtree. Any embedding that is generated in this way will preserve infima in the tree order. As we demand $m=n$ in~(ii), our embeddings preserve labels. The condition~$r(s)\leq n$ in~(iii) amounts to the strong gap condition of Harvey Friedman (see~\cite{simpson85}). Indeed, it recursively ensures the following: All nodes in~$t$ that lie below the image of the root of~$s$ will have label at least~$r(s)$. Via clause~(ii), the analogous condition propagates to all `gaps' that the image of~$s$ leaves in~$t$. These claims are verified in the proof of~\cite[Proposition~5.3]{freund-kruskal-gap}. Part~(a) of the following exercise should make them plausible.

\begin{exercise}\label{ex:gap-partial-order}
(a) Find examples for $s\gap t$ and for~$s\not\gap t$. Draw these examples as labelled trees and embeddings.

(b) Show that $(M(Y),\leq^M)$ is a partial order when the same holds for~$(Y,\leq)$. \emph{Remark:} The proof of antisymmetry deserves some thought.

(c) Consider the height function $h:T_N(X)\to\mathbb N$ with
\begin{equation*}
h(\overline x):=0\quad\text{and}\quad h(n\star[t_0,\ldots,t_{k-1}]):=\max\big(\{0\}\cup\{h(t_i)+1\,|\,i<k\}\big).
\end{equation*}
Show that $s\gap t$ implies~$h(s)\leq h(t)$.

(d) Prove that $\gap$ is a partial order on~$T_N(X)$. \emph{Hint:} Use~(c) to show that $t\gap s$ fails when $s\gap t$ holds by clause~(iii) of Definition~\ref{def:gap-order}.
\end{exercise}

The following is copied from~\cite{freund-kruskal-gap} but already implicit in~\cite[Section~4]{simpson85}. We point out that elements of $T_N(T_{N+1}^0(X))$ have the form $\overline s$ with $s\in T_{N+1}^0(X)\subseteq T_{N+1}(X)$ or the form $n\star[t_0,\ldots,t_{k-1}]$ with $n<N$ and $t_i\in T_N(T_{N+1}^0(X))$.

\begin{definition}
Given $N\in\mathbb N$ and a partial order~$X$, we recursively define
\begin{equation*}
\pi:T_N\left(T_{N+1}^0(X)\right)\to T_{N+1}(X)
\end{equation*}
by $\pi(\overline s):=s$ and $\pi(n\star[t_0,\ldots,t_{k-1}]):=(n+1)\star[\pi(t_0),\ldots,\pi(t_{k-1})]$. Then let
\begin{equation*}
\kappa:M\left(T_N\left(T_{N+1}^0(X)\right)\right)\to T^0_{N+1}(X)
\end{equation*}
be given by $\kappa([t_0,\ldots,t_{k-1}]):=0\star[\pi(t_0),\ldots,\pi(t_{k-1})]$. Finally, we assign a finite set $\supp(t)\subseteq T_{N+1}^0(X)$ to each element~$t\in T_N(T_{N+1}^0(X))$, by setting $\supp(\overline s):=\{s\}$ and $\supp(n\star\tau):=\supp^M(\tau)$ with $\supp^M([t_0,\ldots,t_{k-1}]):=\bigcup_{i<k}\supp(t_i)$.
\end{definition}

It is instructive to consider the case where we have $N=0$ and~$X=\emptyset$. Here the set $T_1(\emptyset)=T_1^0(\emptyset)$ and the relation~$\gap$ coincide with the set of finite rooted trees without labels and the usual embeddability relation, as the single label in~$1=\{0\}$ has no effect. The $\pi$ above is the obvious isomorphism $T_0(Y)\cong Y$ for~$Y=T_1(\emptyset)$. Modulo the latter, our $\kappa$ amounts to
\begin{equation*}
M\left(T_1(\emptyset)\right)\to T_1(\emptyset)\quad\text{with}\quad [t_0,\ldots,t_{k-1}]\mapsto 0\star[t_0,\ldots,t_{k-1}].
\end{equation*}
This map corresponds to the recursive construction of trees, which yields a tree with a new root (labelled by~$0$) when the immediate subtrees~$t_i$ are already constructed. We point out that $0\star[t_0,\ldots,t_{k-1}]$ will thus depend on the elements of the `support' set $\supp^M([t_0,\ldots,t_{k-1}])=\{t_0,\ldots,t_{k-1}\}$ (still computed modulo~$T_0(Y)\cong Y$). The recursive construction of trees corresponds to the fact that $T_1(\emptyset)=T_1^0(\emptyset)$ is the least fixed point of the transformation
\begin{equation*}
Y\mapsto M(Y)\cong M(T_0(Y)).
\end{equation*}
More generally, one can view $T_{N+1}^0(\emptyset)$ as the least fixed point of~$Y\mapsto M(T_N(Y))$. Once we have $T_{N+1}^0(\emptyset)$ and $Y\mapsto T_N(Y)$, we reach $T_{N+1}(\emptyset)$ via the map~$\pi$ (which is non-trivial for~$N>0$). This yields a recursive construction that is made precise in~\cite{freund-kruskal-gap}, based on earlier work by Ryu Hasegawa~\cite{hasegawa97,hasegawa-analytic} and Andreas Weiermann~\cite{weiermann09}. Part~(a) of the following exercise provides vital intuition. For a solution of (b)~to~(d), we refer to Proposition~6.2 and Lemma~6.3 of~\cite{freund-kruskal-gap} (cf.~also~\cite[Lemma~4.5]{simpson85}).

\begin{exercise}
(a) Consider some examples of trees in~$T_2(\emptyset)$ and~$T^0_2(\emptyset)$. Determine their preimages under~$\pi$ and under~$\kappa$, respectively.

(b) Prove that $\pi:T_N\left(T_{N+1}^0(X)\right)\to T_{N+1}(X)$ is an order isomorphism for any partial order~$X$ and any~$N\in\mathbb N$. \emph{Remark:} Use different symbols such as~$\gap$ and $\gap_0$ for the order relations on~$T_{N+1}(X)\supseteq T^0_{N+1}(X)$ and on $T_N(T_{N+1}^0(X))$, which arise from iterative applications of Definition~\ref{def:gap-order}.

(c) Show that any $s\in T^0_{N+1}(X)$ and $t\in T_N(T_{N+1}^0(X))$ validate
\begin{equation*}
s\gap\pi(t)\quad\Leftrightarrow\quad s\gap t'\text{ for some }t'\in\supp(t).
\end{equation*}
\emph{Remark:} Intuitively, the equivalence shows that a tree $s$ with root label~$0$ embeds into a tree $\pi(t)$ with root label $n+1>0$ precisely when it embeds into a subtree~$t'$ with root label~$0$. The crucial point is that the gap condition is immaterial below nodes with label~$0$, i.\,e., that the condition $r(s)\leq n$ in Definition~\ref{def:gap-order} is automatic when we have $r(s)\leq 0$. This explains the special status of $T^0_{N+1}(X)\subseteq T_{N+1}(X)$.

(d) Prove that any $\sigma,\tau\in M(T_N(T_{N+1}^0(X)))$ validate
\begin{equation*}
\kappa(\sigma)\gap\kappa(\tau)\quad\Leftrightarrow\quad\sigma\gap_0^M\tau\text{ or }\kappa(\sigma)\gap t'\text{ for some }t'\in\supp^M(\tau),
\end{equation*}
for $\gap$ and $\gap_0$ as in the remark in part~(b). Also show that $\overline x\gap\kappa(\tau)$ holds precisely when we have $\overline x\gap t'$ for some $t'\in\supp^M(\tau)$. \emph{Remark:} It may be instructive to~consider the case of~$N=0$. Here the equivalence is just the usual recursive characterization of tree embeddability (note that $\gap$ and $\gap_0$ coincide modulo~$T_0(Y)\cong Y$).
\end{exercise}

An infinite sequence $y_0,y_1,\ldots$ in a partial order~$(Y,\leq_Y)$ is called bad if there are no indices~$i<j$ with $y_i\leq_Y y_j$. One says that $Y$ (together with~$\leq_Y$) is a well partial order if there are no infinite bad sequences. As we have seen, the usual embeddability relation between unlabelled trees coincides with our relation $\gap$ on~$T_1(\emptyset)$. The~statement that this relation is a well partial order is known as Kruskal's theorem. In view of the discussion above, the following can be seen as an iterated version of that theorem (where the labels keep track of the iterations). Even though we are most interested in the result for~$X=\emptyset$, the inclusion of arbitrary~$X$ is essential for the recursive construction and the following proof by induction.

\begin{theorem}[{Harvey Friedman's extended Kruskal theorem; see~\cite{simpson85}}]\label{thm:EKT}
For each number~$N\in\mathbb N$, if $X$ is a well partial order, then so is~$T_N(X)$.
\end{theorem}
\begin{proof}
We argue by induction on~$N$. In the base case it suffices to recall~$T_0(X)\cong X$ from above. For the induction step we fix a well partial order~$X$. By the induction hypothesis and part~(b) of the previous exercise, it suffices to show that $T_{N+1}^0(X)$ is a well partial order. Let us assume that this fails. We then define $B\neq\emptyset$ as the set of all finite sequences $\langle t_0,\ldots,t_{i-1}\rangle$ that are initial segments of some infinite bad sequence in~$T^0_{N+1}(X)$. Let us observe that any sequence in~$B$ can be extended. Consider the height function $h:T_{N+1}(X)\to\mathbb N$ from Exercise~\ref{ex:gap-partial-order}(c). Recursively, we pick $t_i$ such that $h(t_i)$ is minimal with~$\langle t_0,\ldots,t_i\rangle\in B$, where $t_0,\ldots,t_{i-1}$ are the elements that have already been picked. This yields a bad sequence $t_0,t_1,\ldots$ that is minimal, in the sense that no sequence $t_0,\ldots,t_{i-1},t_i',t'_{i+1},\ldots$ with $h(t'_i)<h(t_i)$ can be bad. We point out that the use of minimal bad sequences is due to Crispin \mbox{Nash-Williams}~\cite{nash-williams63}. Any element of~$T^0_{N+1}(X)$ can be uniquely written as~$\overline x$ or as~$\kappa(\sigma)$, since $\pi$ is bijective by the previous exercise. Let us consider the set
\begin{equation*}
Y:=\bigcup\{\supp^M(\sigma)\,|\,t_i=\kappa(\sigma)\text{ for some }i\in\mathbb N\}\subseteq T^0_{N+1}(X).
\end{equation*}
We show that $\gap$ is a well partial order on~$Y$. If not, we find indices $i(0)<i(1)<\ldots$ and a bad sequence $t'_{i(0)},t'_{i(1)},\ldots$ with $t'_{i(j)}\in\supp^M(\sigma)$ for~$t_{i(j)}=\kappa(\sigma)$. Note that we get $h(t'_{i(j)})<h(t_{i(j)})$, as a straightforward induction on~$t$ shows that $s\in\supp(t)$ entails $h(s)\leq h(\pi(t))$. By minimality, it follows that the sequence
\begin{equation*}
t_0,\ldots,t_{i(0)-1},t'_{i(0)},t'_{i(1)},\ldots
\end{equation*}
cannot be bad. We must thus have $t_i\gap t'_{i(j)}$ for some~$i<i(0)$ and some~$j\in\mathbb N$. But then part~(d) of the previous exercise yields~$t_i\gap t_{i(j)}$, which contradicts the fact that $t_0,t_1,\ldots$ is bad. As~$X$ is a well partial order, only finitely many entries of the bad sequence $t_0,t_1,\ldots$ can have the form~$\overline x$. After passing to a subsequence, we may assume that they are all given as~$t_i=\kappa(\sigma_i)$. The choice of~$Y$ yields
\begin{equation*}
\sigma_i\in M(T_N(Y))\subseteq M(T_N(T^0_{N+1}(X))),
\end{equation*}
as an induction on~$t$ shows that $\supp(t)\subseteq Y$ entails~$t\in T_N(Y)$. Given that~$Y$ is a well partial order, the same holds for~$M(T_N(Y))$, by the main induction hypothesis and Higman's lemma (see e.\,g.~\cite[Remark~X.3.19]{simpson09}; the result for multisets reduces to the one for sequences).  We thus find $i<j$ with $\sigma_i\gap_0^M\sigma_j$, in the notation from the previous exercise. By part~(d) of the latter, we get $t_i\gap t_j$, which is incompatible with the assumption that $t_0,t_1,\ldots$ is bad.
\end{proof}

As it stands, the previous result does not remain valid when~$N$ is replaced by an infinite ordinal. Indeed, if we have $n<\alpha$ for all~$n\in\mathbb N$, then we obtain the bad sequence $0\star[],1\star[],\ldots$ in~$T_\alpha(\emptyset)$. On the other hand, if one demands $m\leq n$ rather than $m=n$ in clause~(ii) of Definition~\ref{def:gap-order}, then $T_\alpha(\emptyset)$ becomes a well partial order for any ordinal~$\alpha$, as conjectured by Friedman and shown by Igor K\v{r}\'i\v{z}~\cite{kriz-generalized-kruskal}. Let~us conclude this section with some metamathematical considerations:

\begin{remark}\label{rmk:unprovability}
In the previous proof, the complement of the set~$B$ can be defined by $\Pi^1_1$-comprehension, as the definition of~$B$ involves a single existential quantifier over infinite objects (`some infinite bad sequence'). Once $B$ is available, the minimal bad sequence $t_0,t_1,\ldots$ from the previous proof can be formed by arithmetical comprehension (pick $t_i$ with minimal code to avoid choice). One can conclude that $\pica$ proves the result of Theorem~\ref{thm:EKT} for each externally fixed~$N\in\mathbb N$. The induction principle that yields the result for all $N$ is not available in~$\pica$. To explain this, we recall that induction has only been included for properties that are given by a set variable. With the help of comprehension, we get induction for $\Pi^1_1$-statements (cf.~Exercise~6.4 from the first part of the lecture~\cite{first-course}). The present induction statement does not fall into this class: it involves alternating set quantifiers, as the $\Pi^1_1$-property of `being a well partial order' appears in both premise and conclusion. Theorem~\ref{thm:EKT} is indeed unprovable in the theory~$\pica$, by a famous result of Friedman (see again~\cite{simpson85}). Together with Robertson and Seymour~\cite{friedman-robertson-seymour},~the latter has deduced that $\pica$ cannot prove the graph minor theorem, as mentioned in the introduction. In the aforementioned definition of~$B$, the well partial order~$X$ occurs as a parameter. Hence the parameter-free version~$\pica^-$ can only form~$B$ for concrete $X$ with a suitable definition. This suffices to show that $T_1(\emptyset)$ is a well partial order, so that $\pica^-$ proves the original Kruskal theorem. To prove that $T_2(\emptyset)$ is a well partial order, on the other hand, we would need to consider $T_1(Y)$ for a `complex' set~$Y$. In Corollary~\ref{cor:EKT-unprovable} below, we will show that $\pica^-$ cannot prove the result that $T_2(\emptyset)$ is a well partial order.
\end{remark}

\section{The Bachmann-Howard ordinal}\label{sect:Bachmann-Howard}

The strength of $\pica^-$ is measured by the so-called Bachmann-Howard ordinal. In this section, we introduce this ordinal via a notation system $\vartheta(\varepsilon_{\Omega+1})$ that is taken from~\cite{rathjen-weiermann-kruskal}. We also show that $\bho$ can be embedded into the tree order~$T_2(\emptyset)$ that was discussed above. In view of Theorem~\ref{thm:EKT}, this confirms that $\bho$ is a well order. Conversely, we will eventually deduce that the cited theorem for~$N=2$ is unprovable in~$\pica^-$. Let us note that one could prove a somewhat sharper result by working with a certain suborder of~$T_2(\emptyset)$, which has been determined by Jeroen van der Meeren, Michael Rathjen and Andreas Weiermann~\cite{MRW-Bachmann} (for the case of ordered trees).

A detailed justification of the following recursion is given below. We will also motivate the definition in terms of set theory and combinatorics.

\begin{definition}\label{def:Bachmann-Howard}
By recursion, we simultaneously define a set $\vartheta(\varepsilon_{\Omega+1})$ of terms, a binary relation~$\prec$ on this set, and a function~$E$ that assigns a finite $E(\alpha)\subseteq\vartheta(\varepsilon_{\Omega+1})$ to each $\alpha\in\vartheta(\varepsilon_{\Omega+1})$. First, the set $\vartheta(\varepsilon_{\Omega+1})$ is generated as follows:
\begin{enumerate}[label=(\roman*)]
\item we have a term $\Omega\in\vartheta(\varepsilon_{\Omega+1})$,
\item for each $\alpha\in\vartheta(\varepsilon_{\Omega+1})$ we get another term~$\vartheta\alpha\in\vartheta(\varepsilon_{\Omega+1})$,
\item we get a term $\langle\alpha_0,\ldots,\alpha_{n-1}\rangle$ for any $\alpha_0,\ldots,\alpha_{n-1}\in\vartheta(\varepsilon_{\Omega+1})$ that meet the following conditions (note that $n=0$ is also permitted):
\begin{itemize}[label={--}]
\item if $n>1$ then we have $\alpha_{n-1}\preceq\ldots\preceq\alpha_0$, where $\alpha\preceq\beta$ abbreviates the disjunction of~$\alpha\prec\beta$ and $\alpha=\beta$ (equality as terms),
\item if $n=1$ then the term $\alpha_0$ does not have the form~$\Omega$ or~$\vartheta\beta$.
\end{itemize}
\end{enumerate}
Secondly, the function~$E$ is explained by
\begin{equation*}
E(\Omega):=\emptyset,\quad E(\vartheta\alpha):=\{\vartheta\alpha\},\quad E(\langle\alpha_0,\ldots,\alpha_{n-1}\rangle):=\textstyle\bigcup_{i<n}E(\alpha_i).
\end{equation*}
Finally, the relation $\alpha\prec\beta$ holds precisely when one of the following clauses applies:
\begin{enumerate}[label=(\roman*')]
\item $\alpha=\Omega$ and $\beta=\langle\beta_0,\ldots,\beta_{n-1}\rangle$ with $\Omega\preceq\beta_0$ (in particular $n>0$),
\item $\alpha=\vartheta \alpha'$ and one of the following holds:
\begin{itemize}[label={--}]
\item $\beta=\Omega$ or $\beta=\langle\beta_0,\ldots,\beta_{n-1}\rangle$ with $\alpha\preceq\beta_0$,
\item $\beta=\vartheta\beta'$ with $\alpha'\prec\beta'$ and $\gamma\prec\beta$ for all $\gamma\in E(\alpha')$,
\item $\beta=\vartheta\beta'$ and $\alpha\preceq\gamma$ for some $\gamma\in E(\beta')$,
\end{itemize}
\item $\alpha=\langle\alpha_0,\ldots,\alpha_{m-1}\rangle$ and one of the following holds:
\begin{itemize}[label={--}]
\item $\beta$ is of the form $\Omega$ or $\vartheta\beta'$ and we have $m=0$ or $\alpha_0\prec\beta$,
\item $\beta=\langle\beta_0,\ldots,\beta_{n-1}\rangle$ and for some $j\leq\min(m,n)$ we have $\alpha_i=\beta_i$ for all~$i<j$ and either $j=m<n$ or $j<\min(m,n)$ and $\alpha_j\prec\beta_j$.
\end{itemize}
\end{enumerate}
\end{definition}
In order to justify the recursion, one can disentangle the simultaneous definition as follows. First, generate a larger set $\vartheta^0(\varepsilon_{\Omega+1})$ by clauses~(i) to~(iii) above with the condition~$\alpha_{n-1}\preceq\ldots\preceq\alpha_0$ in~(iii) removed. Let $l:\vartheta^0(\varepsilon_{\Omega+1})\to\mathbb N$ be the length function that is given by
\begin{equation*}
l(\Omega):=0,\quad l(\vartheta\alpha):=l(\alpha)+1,\quad l(\langle\alpha_0,\ldots,\alpha_{n-1}\rangle):=n+\textstyle\sum_{i<n}l(\alpha_i).
\end{equation*}
The same clauses as before define~$E$ on the larger set $\vartheta^0(\varepsilon_{\Omega+1})$. A straightforward induction over terms shows that $\alpha'\in E(\alpha)$ implies $l(\alpha')\leq l(\alpha)$. One can now decide $\alpha\in\vartheta(\varepsilon_{\Omega+1})$ and $\beta\prec\gamma$ by simultaneous recursion on $l(\alpha)$ and $l(\beta)+l(\gamma)$.

To provide some first intuition, we note that $\langle\alpha_0,\ldots,\alpha_{n-1}\rangle$ is supposed to represent an ordinal $\omega^{\alpha_0}+\ldots+\omega^{\alpha_{n-1}}$ in Cantor normal form (cf.~Remark~5.5 in the first part of these lecture notes~\cite{first-course}). The terms $\vartheta\alpha$ and~$\Omega$ denote $\varepsilon$-numbers, i.\,e., fixed points of the function~$\gamma\mapsto\omega^\gamma$. Thus $\langle\vartheta\alpha\rangle$ and $\langle\Omega\rangle$ would have the same interpretation as $\vartheta\alpha$ and~$\Omega$, which is why we exclude them in clause~(iii) above. In view of clause~(ii'), our term structure gives rise to a function
\begin{equation*}
\vartheta:\vartheta(\varepsilon_{\Omega+1})\to\vartheta(\varepsilon_{\Omega+1})\cap\Omega:=\{\alpha\in\vartheta(\varepsilon_{\Omega+1})\,|\,\alpha\prec\Omega\}.
\end{equation*}
If $\vartheta(\varepsilon_{\Omega+1})$ is to be a well order, this function into an initial segment cannot be fully order preserving (see, e.\,g.,~\cite[Exercise~3.11]{first-course}). The second point in~(ii') asserts that it is order preserving under the side condition that $\gamma\prec\beta=\vartheta\beta'$ holds for~\mbox{$\gamma\in E(\alpha')$}. In the usual set theoretic interpretation -- which will not play an official role in this lecture --, one would interpret $\Omega$ as the first uncountable or nonrecursive ordinal. The aforementioned side condition has the effect that $\vartheta\alpha'\prec\vartheta\beta'$ is only required for countably many $\alpha'$, so that a suitable value~$\vartheta\beta'\prec\Omega$ is available. A more detailed explanation is given in the paragraph before Proposition~\ref{prop:vartheta-set-def} below. Under the indicated interpretations, the ordinal $\Omega$ itself is the $\Omega$-th $\varepsilon$-number, which is commonly denoted by $\varepsilon_{\Omega}$. The terms in $\vartheta(\varepsilon_{\Omega+1})$ denote ordinals below the next $\varepsilon$-number~$\varepsilon_{\Omega+1}$. Intuitively, the construction has a self-strengthening aspect: Due to the third point in~(ii'), we have $\alpha\prec\vartheta\beta'$ for $\alpha\in E(\beta')$, as seen in the following exercise. This ensures that the values of $\vartheta$ are reasonably large, so that the side condition $\gamma\prec\vartheta\beta'$ is satisfied quite often. But then $\vartheta$ is almost order preserving, which forces its values to be even larger.

\begin{exercise}\label{ex:Bachmann-Howard}
(a) Let us abbreviate
\begin{align*}
E(\alpha)\prec^*\beta\quad&:\Leftrightarrow\quad \alpha'\prec\beta\text{ for all }\alpha'\in E(\alpha),\\
\alpha\preceq^* E(\beta)\quad&:\Leftrightarrow\quad \alpha\preceq \beta'\text{ for some }\beta'\in E(\beta).
\end{align*}
Observe that clause~(ii') above yields
\begin{equation*}
\vartheta\alpha\prec\vartheta\beta\quad\Leftrightarrow\quad(\alpha\prec\beta\text{ and }E(\alpha)\prec^*\vartheta\beta)\text{ or }\vartheta\alpha\preceq^*E(\beta).
\end{equation*}
Convince yourself that we get $E(\alpha)\prec^*\vartheta\alpha$ for any $\alpha\in\vartheta(\varepsilon_{\Omega+1})$.

(b) Prove that $\prec$ is a linear order on $\vartheta(\varepsilon_{\Omega+1})$. \emph{Hint:} For irreflexivity, you will need to show that $\vartheta\alpha\preceq^* E(\alpha)$ fails. If you prove transitivity first, you can use (a) to conclude inductively.

(c) For $\alpha\prec\Omega$, show that $E(\alpha)\prec^*\vartheta\beta$ is equivalent to $\alpha\prec\vartheta\beta$. \emph{Remark:} The set $E(\alpha)$ consists of the $\varepsilon$-numbers below~$\Omega$ in the hereditary Cantor normal form~of~$\alpha$.

(d) Prove that we always have $\alpha_0\prec\langle\alpha_0,\ldots,\alpha_{n-1}\rangle$. Also prove that we get
\begin{equation*}
\langle\alpha_0,\ldots,\alpha_{m-1}\rangle\preceq\langle\beta_0,\ldots,\beta_{n-1}\rangle
\end{equation*}
when there is an injection $g:\{0,\ldots,m-1\}\to\{0,\ldots,n-1\}$ such that $\alpha_i\preceq\beta_{g(i)}$ holds for all~$i<m$. \emph{Hint:} If $g(i)=i$ for $i<j$ but $g(j)>j=g(k)$ with $j<k$, then we have $\alpha_j\preceq\beta_{g(j)}\preceq\beta_{g(k)}$ and $\alpha_k\preceq\alpha_j\preceq\beta_{g(j)}$. Thus you can swap $g(j)$ and~$g(k)$ to achieve $g(j)=j$.

(e) Show that we have
\begin{equation*}
\alpha\preceq\beta\prec\Omega\quad\Rightarrow\quad\text{for any~$\gamma\in E(\alpha)$ there is a $\delta\in E(\beta)$ with $\gamma\preceq\delta$}.
\end{equation*}
Observe that the implication can fail when the condition~$\beta\prec\Omega$ is removed.
\end{exercise}

The following map and Theorem~\ref{thm:EKT} will ensure that~$\vartheta(\varepsilon_{\Omega+1})$ is a well order.

\begin{definition}\label{def:emb-T2}
Let $f:\vartheta(\varepsilon_{\Omega+1})\to T_2(\emptyset)=:T_2$ be given by the recursive clauses
\begin{gather*}
f(\Omega):=1\star[],\quad f(\vartheta\alpha):=0\star[1\star[f(\alpha)]],\\
f(\langle\alpha_0,\ldots,\alpha_{n-1}\rangle):=i\star[f(\alpha_0),\ldots,f(\alpha_{n-1})]\text{ with }i=\begin{cases}
0&\text{if }n=0\text{ or }\alpha_0\prec\Omega,\\
1&\text{otherwise}.
\end{cases}
\end{gather*}
\end{definition}

Our aim is to show that $f$ is order reflecting. As preparation, we construct a version of~$E$ on the level of $T_2$.

\begin{definition}
For $t\in T_2$ we recursively define $\overline E(t)\subseteq T_2$ by the clause
\begin{equation*}
\overline E(t):=\begin{cases}
\{t\}\quad\text{if $t=0\star[i_0\star\tau_0,\ldots,i_{n-1}\star\tau_{n-1}]$ with $i_j=1$ for some~$j<n$},\\
\textstyle\bigcup_{j<n}\overline E(t_j)\quad\text{if $t=i\star[t_0,\ldots,t_{n-1}]$ is of a different form}.
\end{cases}
\end{equation*}
\end{definition}

Let us collect some basic properties:

\begin{lemma}\label{lem:E-bar}
(a) We have $\overline E(f(t))=f[E(t)]:=\{f(s)\,|\,s\in E(t)\}$.

(b) For any $s\in\overline E(t)$ we have $s\gap t$.

(c) If we have $f(\vartheta s)\gap t$, then we get $f(\vartheta s)\gap t'$ for some~$t'\in\overline E(t)$.
\end{lemma}
\begin{proof}
(a) To conclude by a straightforward induction over~$t$, it suffices to show
\begin{equation*}
\overline E(f(\langle t_0,\ldots,t_{n-1}\rangle))=\textstyle\bigcup_{j<n}\overline E(f(t_j)).
\end{equation*}
This could only fail if we had $f(\langle t_0,\ldots,t_{n-1}\rangle):=0\star[f(t_0),\ldots,f(t_{n-1})]$ while some $f(t_j)$ was of the form~$1\star\tau_j$. One readily checks that this would give $\Omega\preceq t_j\preceq t_0$. But then the definition of~$f$ would yield a $1$ at the place of our~$0$.

(b) By the definition of~$\overline E$, any term $s\in\overline E(t)$ is of the form $0\star\sigma$, so that the condition $r(s)=0\leq n$ from part~(iii) of Definition~\ref{def:gap-order} is automatic. In view of this fact, the claim is readily checked by induction on~$t$.

(c) We argue by induction on~$t=i\star[t_0,\ldots,t_{n-1}]$. If we have $f(\vartheta s)\gap t_j$, then we inductively get $f(\vartheta s)\gap t'$ for some $t'\in\overline E(t_j)$. The claim follows because we have either $t\in\overline E(t)$ or $\overline E(t_j)\subseteq\overline E(t)$. Now assume that $f(\vartheta s)=0\star[1\star[f(s)]]\gap t$ holds because we have $i=0$ and $1\star[f(s)]\gap t_j$ for some~$j<n$. Here $t_j$ must be of the form $1\star\tau_j$, so that we get $\overline E(t)=\{t\}$, which makes the claim trivial.
\end{proof}

As promised, we now derive that $f$ is order reflecting.

\begin{theorem}\label{thm:emb-T2}
For any $\alpha,\beta\in\vartheta(\varepsilon_{\Omega+1})$ with $f(\alpha)\gap f(\beta)$ we get $\alpha\preceq\beta$.
\end{theorem}
\begin{proof}
As preparation, we observe that $1\star[s_0,\ldots,s_{m-1}]\gap t$ in $T_2$ entails $s_j\gap t$ for all~$j<m$. To see this, write $t=i\star[t_0,\ldots,t_{n-1}]$ and note that the assumption forces~$1\leq i$. In $T_2$ we always have $r(s_j)\leq 1$, so that we get $s_j\gap t$ whenever we have~$s_j\gap t_k$. In view of this fact, the preparatory claim follows by a straightforward induction on~$t$. Also as preparation, we prove
\begin{equation*}
f(\alpha)\gap f(\beta)\quad\Rightarrow\quad l(\alpha)\leq l(\beta),
\end{equation*}
for the  length function that was specified after Definition~\ref{def:Bachmann-Howard}. Both this implication and the claim of the theorem are established by induction on $l(\alpha)+l(\beta)$ and a case distinction according to the forms of~$\alpha,\beta\in\vartheta(\varepsilon_{\Omega+1})$. For convenience, we present the proofs of $l(\alpha)\leq l(\beta)$ and $\alpha\preceq\beta$ at the same time, even though the inductive proof of $l(\alpha)\leq l(\beta)$ does officially come first. In the most interesting case, we have
\begin{equation*}
f(\alpha)=f(\vartheta\alpha')=0\star[1\star[f(\alpha')]]\gap 0\star[1\star[f(\beta')]]=f(\vartheta\beta')=f(\beta).
\end{equation*}
According to Definition~\ref{def:gap-order}, this inequality can hold for two reasons: First assume that we have $f(\alpha)\gap 1\star[f(\beta')]$, which forces $f(\alpha)\gap f(\beta')$ due to the root labels. Inductively, we already get the claim $l(\alpha)\leq l(\beta')<l(\beta)$ about lengths. Furthermore, Lemma~\ref{lem:E-bar} yields $f(\alpha)\gap r$ for some term $r\in\overline E(f(\beta'))=f[E(\beta')]$, which we can thus write as $r=f(\gamma)$ with $\gamma\in E(\beta')$. As before we have $l(\gamma)\leq l(\beta')$, so that the induction hypothesis gives $\alpha\preceq\gamma$. We have thus established $\alpha=\vartheta\alpha'\preceq^* E(\beta')$, which yields $\alpha\prec\vartheta\beta'=\beta$ by Definition~\ref{def:Bachmann-Howard} (see also Exercise~\ref{ex:Bachmann-Howard}). We now assume that $f(\vartheta\alpha')\gap f(\vartheta\beta')$ holds because we have $1\star[f(\alpha')]\gap 1\star[f(\beta')]$. A priori, the latter can be due to $f(\alpha')\gap f(\beta')$ or to $1\star[f(\alpha')]\gap f(\beta')$. By the observation at the beginning of the proof, however, we get $f(\alpha')\gap f(\beta')$ in any case. We can conclude $l(\alpha)\leq l(\beta)$ as well as $\alpha'\preceq\beta'$. If the latter is an equality, so is~$\alpha\preceq\beta$. Now assume that we have $\alpha'\prec\beta'$. To get $\alpha\prec\beta$, we show $E(\alpha')\prec^*\beta$. For any $\gamma\in E(\alpha')$, Lemma~\ref{lem:E-bar} yields $f(\gamma)\in f[E(\alpha')]=\overline E(f(\alpha'))$ and then $f(\gamma)\gap f(\alpha')\gap f(\beta')$. Inductively, we obtain $l(\gamma)\leq l(\beta')<l(\beta)$, which means that $\gamma$ and $\beta$ are different. Thus it suffice to show $\gamma\preceq\beta$ in order to get $\gamma\prec\beta$ (which is the only point of the auxiliary claim about lengths). From $f(\gamma)\gap f(\beta')$ we get $f(\gamma)\gap f(\beta)$, since $f(\gamma)\in E(f(\alpha'))$ is of the form~$0\star\rho$ (cf.~the proof of Lemma~\ref{lem:E-bar}). We can now infer $\gamma\preceq\beta$ by the induction hypothesis. Let us also consider an inequality
\begin{multline*}
f(\alpha)=f(\vartheta\alpha')=0\star[1\star[f(\alpha')]]\gap\\ i\star[f(\beta_0),\ldots,f(\beta_{n-1})]=f(\langle\beta_0,\ldots,\beta_{n-1}\rangle)=f(\beta).
\end{multline*}
It is straightforward to deduce $l(\alpha)\leq l(\beta)$ from the induction hypothesis, given that $1\star[f(\alpha')]\gap f(\beta_j)$ entails $f(\alpha')\gap f(\beta_j)$. When we have $i=1$, the definition of~$f$ and Exercise~\ref{ex:Bachmann-Howard}(d) yield $\alpha\prec\Omega\preceq\beta_0\prec\beta$. Now assume that we have $i=0$ and hence~$\beta_{n-1}\preceq\ldots\preceq\beta_0\prec\Omega$. Then $f(\beta_j)$ is of the form $0\star\tau_j$, as in the proof of Lemma~\ref{lem:E-bar}(a). This entails $1\star[f(\alpha')]\not\gap f(\beta_j)$, so that we must have $f(\alpha)\gap f(\beta_j)$ for some~$j<n$. We can conclude $\alpha\preceq\beta_j\preceq\beta_0\prec\beta$ by the induction hypothesis and Exercise~\ref{ex:Bachmann-Howard}. Finally, we consider the case of an inequality
\begin{multline*}
f(\alpha)=f(\langle\alpha_0,\ldots,\alpha_{m-1}\rangle)=i\star[f(\alpha_0),\ldots,f(\alpha_{m-1})]\gap\\ 0\star[1\star[f(\beta')]]=f(\vartheta\beta')=f(\beta).
\end{multline*}
The latter can only hold if we have~$i=0$ and thus $\alpha_0\prec\Omega$, which entails $\alpha\prec\Omega$. To get $\alpha\prec\vartheta\beta'=\beta$ via Exercise~\ref{ex:Bachmann-Howard}(c), it suffices to show that we have $E(\alpha_0)\prec^*\beta$. Given $\gamma\in E(\alpha_0)$, we first observe $l(\gamma)\leq l(\alpha_0)<l(\alpha)\leq l(\beta)$. As above, this means that we need only show $\gamma\preceq\beta$ rather than $\gamma\prec\beta$. Once again, we invoke Lemma~\ref{lem:E-bar} to get $f(\gamma)\in f[E(\alpha)]=\overline E(f(\alpha))$ and then $f(\gamma)\gap f(\alpha)\gap f(\beta)$. Now the induction hypothesis yields $\gamma\preceq\beta$. The remaining cases are similar and easier.
\end{proof}

As indicated above, we can draw the following conclusion.

\begin{corollary}\label{cor:bachmann-howard-wo}
The order $\vartheta(\varepsilon_{\Omega+1})$ is well founded.
\end{corollary}
\begin{proof}
Given an infinite sequence $\alpha_0,\alpha_1,\ldots$ in $\vartheta(\varepsilon_{\Omega+1})$, we obtain another sequence $f(\alpha_0),f(\alpha_1),\ldots$ in $T_2=T_2(\emptyset)$. The latter is a well partial order by Theorem~\ref{thm:EKT}. Hence we find $i<j$ with $f(\alpha_i)\gap f(\alpha_j)$. By the previous theorem we get~$\alpha_i\preceq\alpha_j$, so that the original sequence in $\vartheta(\varepsilon_{\Omega+1})$ is not strictly descending.
\end{proof}

In the following, some aspects of the set theoretic construction of~$\vartheta(\varepsilon_{\Omega+1})$ are recovered on a syntactical level.

\begin{definition}\label{def:coeff-sets}
For $\alpha,\beta\in\vartheta(\varepsilon_{\Omega+1})$, we declare that $C_\alpha(\beta)\subseteq\vartheta(\varepsilon_{\Omega+1})$ is generated by the following recursive clauses:
\begin{enumerate}[label=(\roman*)]
\item we have $\Omega\in C_\alpha(\beta)$ as well as $\gamma\in C_\alpha(\beta)$ for all $\gamma\prec\beta$,
\item given $\gamma\in C_\alpha(\beta)$ with $\gamma\prec\alpha$, we get $\vartheta\gamma\in C_\alpha(\beta)$,
\item we get $\langle\gamma_0,\ldots,\gamma_{n-1}\rangle\in C_\alpha(\beta)$ whenever we have $\gamma_i\in C_\alpha(\beta)$ for all~$i<n$.
\end{enumerate}
\end{definition}

Let us stress that the definition of $C_\alpha(\beta)$ refers to $\vartheta\gamma$ for $\gamma\prec\alpha$ only. Conversely, the following proposition shows that the sets $C_\alpha(\beta)$ determine~$\vartheta\alpha$. Working in set theory, one can exploit these observations to construct (sets of) ordinals~$C_\alpha(\beta)$ and~$\vartheta\alpha$ by simultaneous recursion. In order to find a $\gamma=\sup_{n\prec\omega}\gamma_n\prec\Omega$ as in the minimum below, one would first ensure $E(\alpha)\prec^*\gamma_0$ to get $\alpha\in C_\alpha(\gamma_0)$. One would then choose $\gamma_{n+1}\prec\Omega$ with $C_\alpha(\gamma_n)\cap\Omega\prec^*\gamma_{n+1}$, which is possible when $\Omega$ is the first uncountable cardinal, as $C_\alpha(\gamma_n)$ will be countable. We refer to~\cite{rathjen-weiermann-kruskal} for full details of the set theoretic construction, which will not be needed in the following.

\begin{proposition}\label{prop:vartheta-set-def}
We have
\begin{gather*}
\gamma\in C_\alpha(\beta)\quad\Leftrightarrow\quad E(\gamma)\subseteq C_\alpha(\beta),\\
\vartheta\alpha=\min\{\gamma\in\vartheta(\varepsilon_{\Omega+1})\,|\,C_\alpha(\gamma)\cap\Omega\prec^*\gamma\text{ and }\alpha\in C_\alpha(\gamma)\},
\end{gather*}
where $C_\alpha(\gamma)\cap\Omega\prec^*\gamma$ asserts that $\delta\prec\gamma$ holds for any $\delta\in C_\alpha(\gamma)$ with $\delta\prec\Omega$.
\end{proposition}
\begin{proof}
In the equivalence, the implication from right to left can be verified by a straightforward induction on the term~$\gamma$ (recall in particular $E(\vartheta\gamma')=\{\vartheta\gamma'\}$). For the other direction, we argue by induction over the recursive definition of~$C_\alpha(\beta)$. To cover the case where $\gamma\in C_\alpha(\beta)$ is due to~$\gamma\prec\beta$, we note that $\delta\preceq\gamma$ holds for any~$\delta\in E(\gamma)$. This is readily checked by induction on the term~$\gamma$, based on part~(d) of Exercise~\ref{ex:Bachmann-Howard}. Next, we show that $\vartheta\alpha$ is one of the $\gamma$ over which the minimum is taken. According to part~(a) of the cited exercise we have $E(\alpha)\prec^*\vartheta\alpha$, so that clause~(i) of Definition~\ref{def:coeff-sets} yields $E(\alpha)\subseteq C_\alpha(\vartheta\alpha)$. We now get $\alpha\in C_\alpha(\vartheta\alpha)$ by the equivalence that we have just proved. In order to show
\begin{equation*}
\delta\in C_\alpha(\vartheta\alpha)\cap\Omega\quad\Rightarrow\quad\delta\prec\vartheta\alpha,
\end{equation*}
we use induction over the length~$l(\delta)$. Consider the crucial case of a term $\delta=\vartheta\delta'$. If $\delta\in C_\alpha(\vartheta\alpha)$ holds by clause~(i) of Definition~\ref{def:coeff-sets}, then $\delta\prec\vartheta\alpha$ is immediate. Otherwise clause~(ii) applies, which means that we have $\delta'\in C_\alpha(\vartheta\alpha)$ and $\delta'\prec\alpha$. The former entails $E(\delta')\subseteq C_\alpha(\vartheta\alpha)$, so that we obtain $E(\delta')\prec^*\vartheta\alpha$ by the induction hypothesis. Together with $\delta'\prec\alpha$ this yields $\delta=\vartheta\delta'\prec\vartheta\alpha$, again by part~(a) of our exercise. Finally, we show that $\vartheta\alpha\preceq\gamma$ holds for arbitrary $\gamma$ with $C_\alpha(\gamma)\cap\Omega\prec^*\gamma$ and $\alpha\in C_\alpha(\gamma)$. It is enough to prove
\begin{equation*}
\delta\prec\vartheta\alpha\quad\Rightarrow\quad\delta\in C_\alpha(\gamma).
\end{equation*}
We argue by induction over~$l(\delta)$ and consider the crucial case of a term~$\delta=\vartheta\delta'$. First assume that $\vartheta\delta'\prec\vartheta\alpha$ holds because we have~$\delta'\prec\alpha$ and $E(\delta')\prec^*\vartheta\alpha$. Inductively, we learn that $C_\alpha(\gamma)$ contains the elements of~$E(\delta')$, hence $\delta'$ itself and thus also~$\vartheta\delta'$. In the remaining case, we have $\delta=\vartheta\delta'\prec\vartheta\alpha$ due to $\delta\preceq^*E(\alpha)$. Here we argue that $\alpha\in C_\alpha(\gamma)$ entails $E(\alpha)\subseteq C_\alpha(\gamma)\cap\Omega\prec^*\gamma$, so that transitivity yields $\delta\prec\gamma$ and in particular $\delta\in C_\alpha(\gamma)$.
\end{proof}

To conclude this section, we present certain closure operators~$\mathcal H_\gamma$ that were first considered by Wilfried Buchholz~\cite{buchholz-local-predicativity}. They allow for a particularly elegant ordinal analysis of impredicative axiom systems, as we will see in Sections~\ref{sect:operator-control} and~\ref{sect:embedding}.

\begin{definition}\label{def:H_gamma}
For each element $\alpha$ and each subset $X$ of $\vartheta(\varepsilon_{\Omega+1})$, we put
\begin{equation*}
\mathcal H_\alpha(X):=\textstyle\bigcap\left\{C_\gamma(\delta)\,\left|\,\gamma,\delta\in\vartheta(\varepsilon_{\Omega+1})\text{ with }\alpha\prec\gamma\text{ and }X\subseteq C_\gamma(\delta)\right.\right\}\subseteq\vartheta(\varepsilon_{\Omega+1}),
\end{equation*}
with $\mathcal H_\alpha(X)=\vartheta(\varepsilon_{\Omega+1})$ when the index set of the intersection is empty.
\end{definition}

Let us record some fundamental properties (cf.~\cite[Lemma~4.7]{buchholz-local-predicativity}):

\begin{proposition}\label{prop:H-basic}
(a) Given $\alpha\prec\beta$, we get $\mathcal H_\alpha(X)\subseteq\mathcal H_\beta(X)$ for all~$X\subseteq\vartheta(\varepsilon_{\Omega+1})$.

(b) From $\alpha\in\mathcal H_\beta(X)$ and $\alpha\preceq\beta$, we can infer $\vartheta\alpha\in\mathcal H_\beta(X)$.

(c) If we have $X\subseteq\bigcap\{C_\gamma(\vartheta\gamma)\,|\,\alpha\prec\gamma\}$, then we obtain
\begin{equation*}
\alpha\preceq\beta\prec\gamma\text{ and }\beta\in\mathcal H_\alpha(X)\quad\Rightarrow\quad\vartheta\beta\prec\vartheta\gamma.
\end{equation*}
\end{proposition}
\begin{proof}
(a) As $\beta\prec\gamma$ implies $\alpha\prec\gamma$, the intersection for~$\mathcal H_\alpha(X)$ is taken over a larger family of sets $C_\gamma(\delta)$.

(b) It suffices to show that $\vartheta\alpha$ lies in all sets $C_\gamma(\delta)$ from the intersection that yields~$\mathcal H_\beta(X)$. For these we have $\alpha\preceq\beta\prec\gamma$. By clause~(ii) of Definition~\ref{def:coeff-sets}, the claim is thus reduced to $\alpha\in C_\gamma(\delta)$, which follows from~$\alpha\in\mathcal H_\beta(X)$.

(c) By~(a) and~(b) we get $\beta\in\mathcal H_\alpha(X)\subseteq\mathcal H_\beta(X)$ and then $\vartheta\beta\in\mathcal H_\beta(X)\subseteq C_\gamma(\vartheta\gamma)$. In view of~$\vartheta\beta\prec\Omega$, we can invoke Proposition~\ref{prop:vartheta-set-def} to conclude~$\vartheta\beta\prec\vartheta\gamma$.
\end{proof}

We will later see that any $\mathcal H_\alpha(X)$ that contains~$\alpha$ is closed under an order preserving function $\beta\mapsto\alpha+\omega(\beta)$ with $\alpha\prec\alpha+\omega(\beta)$. For $X\subseteq\bigcap\{C_\gamma(\vartheta\gamma)\,|\,\alpha\prec\gamma\}$ as in the proposition above, this will yield a map
\begin{equation*}
\mathcal H_\alpha(X)\ni\beta\mapsto\vartheta(\alpha+\omega(\beta))\in\vartheta(\varepsilon_{\Omega+1})\cap\Omega
\end{equation*}
that is order preserving as well, even though $\mathcal H_\alpha(X)$ contains elements above~$\Omega$. We will use this map to show that certain uncountable proofs can be collapsed into countable ones, as long as the proofs do only involve ordinals that are `controlled' by the closure operators $\mathcal H_\gamma$. Details are provided in Sections~\ref{sect:operator-control} and~\ref{sect:embedding} below.

\section{Inductive definitions and $\Pi^1_1$-comprehension}\label{sect:ind-def}

In this section, we show that $\pica^-$ is conservative over a certain theory $\id_1$ of non-iterated inductive definitions. Several analogous but stronger results have been proved by Solomon Feferman~\cite{feferman70}. Inductive definitions are particularly ameanable to methods of ordinal analysis. For this reason, they have played an important role in the development of the subject, as witnessed by the seminal paper of William Howard~\cite{howard72} and the book by Wilfried Buchholz, Solomon Feferman, Wolfram Pohlers and Wilfried Sieg~\cite{bfps-inductive}. We note that inductive definitions have limitations when it comes to very strong axiom systems. These are often analyzed via systems of set theory, an approach that has been pioneered by Gerhard J\"ager~\cite{jaeger-KPN,jaeger-kripke-platek} (see, e.\,g.,~\cite{rathjen-icm} for a survey of subsequent developments). Specifically, Kripke-Platek set theory offers another elegant way to analyze the theory $\pica^-$ (see~\cite[Chapter~11]{pohlers-proof-theory}). We work with inductive definitions because they have fewer set theoretic prerequisites.

Let us write $\mathcal P(\mathbb N)$ for the powerset of the natural numbers. By an operator we shall mean a map $\Phi:\mathcal P(\mathbb N)\to\mathcal P(\mathbb N)$ that is monotone in the sense that $X\subseteq Y\subseteq\mathbb N$ entails $\Phi(X)\subseteq\Phi(Y)$. For each operator we define
\begin{equation*}
I_\Phi:=\bigcap\{X\subseteq\mathbb N\,|\,\Phi(X)\subseteq X\}.
\end{equation*}
The following exercise shows that our operators provide a reasonable formalization of inductive definitions. Our focus on natural numbers is somewhat arbitrary but inessential, because other finite objects (in particular tuples) can be accommodated via coding. To avoid confusion, we note that the general operators of the present section and the specific closure operators~$\mathcal H_\gamma$ from the previous one are two different notions that should be kept firmly apart (despite some obvious connections).

\begin{exercise}\label{ex:ind-def}
(a) Show that $I_\Phi$ is the least fixed point of the operator~$\Phi$, i.\,e., that we have $\Phi(I_\Phi)=I_\Phi$ and that $I_\Phi\subseteq X$ follows from $\Phi(X)\subseteq X$ (hence in particular from $\Phi(X)=X$). Also observe that a map $\Phi:\mathcal P(\mathbb N)\to\mathcal P(\mathbb N)$ that is not monotone does not need to have any fixed points.

(b) Given a set $\mathcal A\subseteq\mathcal P(\mathbb N)\times\mathbb N$, we define $\Phi[\mathcal A]:\mathcal P(\mathbb N)\to\mathcal P(\mathbb N)$ by
\begin{equation*}
\Phi[\mathcal A](X):=\{x\in\mathbb N\,|\,(A,x)\in\mathcal A\text{ for some }A\subseteq X\}.
\end{equation*}
Note that $\Phi[\mathcal A]$ is monotone. Then show that $I_{\Phi[\mathcal A]}$ is the least set with the property that $A\subseteq I_{\Phi[\mathcal A]}$ and $(A,x)\in\mathcal A$ entail $x\in I_{\Phi[\mathcal A]}$. Given any operator~$\Phi$, we put
\begin{equation*}
\mathcal A_\Phi:=\{(A,x)\in\mathcal P(\mathbb N)\times\mathbb N\,|\,x\in\Phi(A)\}.
\end{equation*}
Show that the operations are inverse in the sense that have $\Phi[\mathcal A_\Phi]=\Phi$ and
\begin{equation*}
\mathcal A_{\Phi[\mathcal A]}=\overline{\mathcal A}:=\{(A,x)\in \mathcal P(\mathbb N)\times\mathbb N\,|\,(A_0,x)\in\mathcal A\text{ for some }A_0\subseteq A\}.
\end{equation*}
Note that we have $\Phi(\overline{\mathcal A})=\Phi(\mathcal A)$ and that $\mathcal A=\overline{\mathcal A_0}$ entails $\overline{\mathcal A}=\mathcal A$, in which case we get $\mathcal A_{\Phi[\mathcal A]}=\mathcal A$.
 \emph{Remark:} The point is that we may view~$\mathcal A$ as a set of `clauses', so that we get a more familiar formulation of inductive definitions. We note that part~(b) of the exercise is taken from Section~6.1 of the textbook by Pohlers~\cite{pohlers-proof-theory}. In part~(d) below, we will say that $\Phi$ is finitary if it can be written as $\Phi[\mathcal A]$ for an $\mathcal A$ that does only contain pairs $(A,x)$ such that~$A$ is finite.

(c) Given a linear order $\tl$ on~$\mathbb N$, define an operator~$\Phi$ such that $I_\Phi$ is the well founded part of~$\tl$. More explicitly, we demand that $I_\Phi$ is the largest initial segment ($x\tl y\in I_\Phi$ implies $x\in I_\Phi$) on which $\tl$ is well founded.

(d) For an operator~$\Phi$, we use recursion along the ordinals to define
\begin{equation*}
I_\Phi^\alpha:=\Phi(I_\Phi^{<\alpha})\quad\text{with}\quad I_\Phi^{<\alpha}:=\textstyle\bigcup_{\gamma<\alpha} I^\gamma_\Phi.
\end{equation*}
Prove that $\alpha<\beta$ entails $I^\alpha_\Phi\subseteq I^\beta_\Phi$. Conclude that we have $I^\alpha_\Phi=I^{<\alpha}_\Phi$ for some countable ordinal~$\alpha$ (and note that this remains valid for all~$\beta\geq\alpha$). Then show that any such $\alpha$ validates $I_\Phi=I^\alpha_\Phi$. To avoid prerequisites from set theory, the reader may focus on finitary~$\Phi$ (see part~(b) above), where we get $I_\Phi=\bigcup_{n\in\mathbb N}I^n_\Phi$. \emph{Remark:}~In a certain sense, the exercise shows that our fixed points admit a construction from below. However, this construction presupposes that we are given a `large' object, namely the collection of all countable ordinals (which is itself uncountable). For this reason, it is considered to be impredicative (cf.~the introduction). Nevertheless, one obtains a construction that is predicative at each stage, or `locally predicative' in the sense of Wolfram Pohlers~\cite{pohlers-local-predicativity}.
\end{exercise}

To ensure monotonicity on a syntactic level, we rely on the notion of positive subformula. Specifically, let $\lpax$ be the language of first order arithmetic with a unary predicate symbol~$X$, which we have already encountered in the first part of this lecture~\cite{first-course}. As in the latter, we officially work with formulas in negation normal form, which are built from negated and unnegated prime formulas by the connectives $\land,\lor$ and $\forall,\exists$. Other connectives will be used as abbreviations. In particular, one obtains the negation~$\neg\varphi$ of a formula~$\varphi$ by applying de Morgan's laws and deleting double negations. As a consequence, the formulas denoted by $\neg\neg\varphi$ and~$\varphi$ are syntactically equal. With respect to the indicated normal form, we declare that an operator form is an $\lpax$-formula that has a single free number variable and no subformulas of the form~$\neg Xt$. As an example, we note that $Xx\to\varphi(x)$ stands for $\neg Xx\lor\varphi(x)$ and is no operator form, which makes sense because positive subformulas of the premise are usually seen as negative. Each $\lpax$-formula~$\varphi=\varphi(x)$ induces $\Phi_\varphi:\mathcal P(\mathbb N)\to\mathcal P(\mathbb N)$~with
\begin{equation*}
\Phi_\varphi(Y):=\{n\in\mathbb N\,|\,(\mathbb N,Y)\vDash\varphi(n)\},
\end{equation*}
where we take $Y$ as the interpretation of the predicate variable~$X$. When $\varphi$ is an operator form, then $\Phi_\varphi$ is monotone, as the reader may check by induction over formulas. The converse does not hold. However, if the monotonicity of $\Phi_\varphi$ is provable in pure logic, then $\varphi$ is equivalent to an operator form, as a consequence of the Craig-Lyndon interpolation theorem (see~\cite[Exercise~6.4.9]{pohlers-proof-theory}). In practice, the natural formalization is usually an operator form.

Let $\lpa$ be the usual language of Peano arithmetic, which does not involve the predicate symbol~$X$ but does otherwise coincide with $\lpax$. We now define $\lid$ as the extension of~$\lpa$ by a unary predicate symbol~$I_\varphi$ for each operator form~$\varphi$. We write the latter as $\varphi(x,X)$ when we wish to display the free number variable and predicate symbol. For an~$\lid$-formula~$\psi(y)$ with a distinguished free variable (and possibly further parameters), we declare that $\varphi(x,\psi)$ is the $\lid$-formula that results from $\varphi$ when each prime formula $Xt$ is replaced by $\psi(t)$. When $\psi$ is the prime formula $I_\varphi y$, we write $\varphi(x,I_\varphi)$ for~$\varphi(x,\psi)$. We can now formulate axiom~schemata
\begin{gather}
\forall x\in\mathbb N\,\big(\varphi(x,I_\varphi)\to I_\varphi x\big),\tag{F}\\
\forall x\in\mathbb N\,\big(\varphi(x,\psi)\to\psi(x)\big)\to\forall x\in\mathbb N\,\big(I_\varphi x\to\psi(x)\big),\tag{L}
\end{gather}
where $\varphi$ and $\psi$ range over operator forms and $\lid$-formulas, respectively. In somewhat intuitive notation, the axioms assert that we have $\Phi_\varphi(I_\varphi)\subseteq I_\varphi$ and that $I_\varphi$ is least with this property, with respect to competitors that are defined by some~$\psi$. We will soon see that the converse implication in~(F) can be derived.

For the definition of Peano arithmetic ($\mathsf{PA}$), we refer to the first part of these lecture notes~\cite{first-course}. Let us now define $\id_1$ as the $\lid$-theory that extends~$\mathsf{PA}$ by the axiom schemata (F) and (L) as well as the equality and induction axioms for all symbols and formulas of the language~$\lid$. The subscript of $\id_1$ indicates that we do not admit iterated inductive definitions, i.\,e., that the predicate symbols~$I_\varphi$ may not occur in operator forms. We now prove the fact that was mentioned above. If the reader has shown $I_\Phi\subseteq\Phi(I_\Phi)$ in Exercise~\ref{ex:ind-def}(a), they will have seen the same argument in somewhat different notation.

\begin{lemma}
For each operator form~$\varphi$ we have $\id_1\vdash\forall x\,\big(\varphi(x,I_\varphi)\leftrightarrow I_\varphi x\big)$.
\end{lemma}
\begin{proof}
The implication from left to right is an instance of axiom schema~(F). We define $\psi(x):=\varphi(x,I_\varphi)$ to write it as $\psi(x)\to I_\varphi x$. As $X$ is positive in $\varphi$, induction over the latter yields $\varphi(x,\psi)\to\varphi(x,I_\varphi)$ or in other words~$\varphi(x,\psi)\to\psi(x)$. By axiom schema (L) we get $I_\varphi x\to\psi(x)$, which is the direction from right to left.
\end{proof}

Our next aim is to show that $\pica^-$ is an extension of~$\id_1$. Some care is needed because $\mathcal L_2$ is no sublanguage of~$\lid$, so that the claim can only hold modulo a suitable translation. As in the previous lecture~\cite{first-course}, we may view $\lpax$ as a sublanguage of $\mathcal L_2$, by treating the predicate symbol~$X$ as a set variable. In particular, we may view each operator form $\varphi(x,X)$ as an $\mathcal L_2$-formula. The condition $\Phi_\varphi(X)\subseteq X$ on the operator defined by~$\varphi$ can be expressed by
\begin{equation*}
\cl_\varphi(X):=\forall x\in\mathbb N\,\big(\varphi(x,X)\to x\in X\big).
\end{equation*}
Analogously, we shall write $\cl_\varphi(\psi)$ for the formula that results when we replace each occurrence of~$t\in X$ by $\psi(t)$, which coincides with the premise of~(L) above. Let us now consider the $\mathcal L_2$-formula
\begin{equation*}
\lf_\varphi(Y):=\forall x\in\mathbb N\,\big(x\in Y\leftrightarrow\forall X\subseteq\mathbb N\,(\cl_\varphi(X)\to x\in X)\big).
\end{equation*}
It asserts that the set variable~$Y$ coincides with the least fixed point~$I_\Phi$, as defined in the paragraph before Exercise~\ref{ex:ind-def}. Let us associate a fixed set variable $Y_\varphi$ with each operator form~$\varphi$. Given an $\lid$-formula $\psi$, we write $\psi^\star$ for the $\mathcal L_2$-formula that we obtain when each prime formula $I_\varphi t$ in~$\psi$ is replaced by $t\in Y_\varphi$. The following shows in particular that $\pica^-$ extends $\id_1$ for $\lpa$-formulas.

\begin{proposition}
For each $\lid$-formula $\psi$ we have
\begin{equation*}
\id_1\vdash\psi\quad\Rightarrow\quad\pica^- + \{\lf_\varphi(Y_\varphi)\,|\,I_\varphi\text{ occurs in }\psi\}\vdash\psi^\star.
\end{equation*}
\end{proposition}

In the following proof, we will see that $\pica^-$ proves the existence of (necessarily unique) sets $Y_\varphi\subseteq\mathbb N$ that satisfy the assumptions $\lf_\varphi(Y_\varphi)$. Hence the latter do not make the theory stronger but merely fix the meaning of the~$Y_\varphi$.

\begin{proof}
Write $\Psi$ for the collection of all operator forms $\varphi$ such that the corresponding instance of the axiom (F) or (L) is used in the given derivation~$\id_1\vdash\psi$. We will use induction over the latter to establish
\begin{equation*}
\aca_0 + \{\lf_\varphi(Y_\varphi)\,|\,\varphi\in\Psi\}\vdash\psi^\star.
\end{equation*}
Let us first show how the proposition can be deduced. The point is that $\Psi$ may contain operator forms $\varphi$ such that $I_\varphi$ does not occur in~$\psi$. To simplify notation, we assume $\Psi=\{\varphi\}$ for such a $\varphi$. From the above we then obtain
\begin{equation*}
\aca_0\vdash\forall Y_\varphi\subseteq\mathbb N\,\big(\lf_\varphi(Y_\varphi)\to\psi^\star\big).
\end{equation*}
Given that the predicate $I_\varphi$ does not occur in~$\psi$, the associated set variable $Y_\varphi$ does not occur in~$\psi^\star$. In order to get $\pica^-\vdash\psi^\star$, it is thus enough to have
\begin{equation*}
\pica^-\vdash\exists Y_\varphi\subseteq\mathbb N.\,\lf_\varphi(Y_\varphi).
\end{equation*}
The latter follows by parameter-free $\Pi^1_1$-comprehension (as the $\lpax$-formula $\varphi$ translates into an arithmetical $\mathcal L_2$-formula with a single free set variable~$X$ that becomes bound by a universal quantifier). In the remaining induction over the derivation of $\psi$ in~$\id_1$, it suffices to consider the case where $\psi$ is an axiom (since the `renaming' of $I_\varphi$ into $Y_\varphi$ does not affect the validity of logical rules). If $\psi$ is an induction axiom in the language $\lid$, then $\psi^\star$ is an instance of arithmetical induction, which is provable in~$\aca_0$ (see, e.\,g., Exercise~6.4 from the first lecture~\cite{first-course}). Here one should note that arithmetical comprehension \emph{with parameters} is admitted in $\aca_0$ and in~$\pica^-$. The equality axioms and the axioms of Robinson arithmetic do not pose any challenge. Let us now assume that $\psi$ is the above instance of (F). Then $\psi^\star$ is the formula $\cl_\varphi(Y_\varphi)$, so that our task is to show
\begin{equation*}
\aca_0\vdash\lf_\varphi(Y_\varphi)\to\cl_\varphi(Y_\varphi).
\end{equation*}
The reader who has proved $\Phi(I_\Phi)\subseteq I_\Phi$ in Exercise~\ref{ex:ind-def}(a) will have seen the relevant argument: Aiming at $\cl_\varphi(Y_\varphi)$, we consider an arbitrary~$x\in\mathbb N$ and assume $\varphi(x,Y_\varphi)$. We need to show $x\in Y_\varphi$, which $\lf_\varphi(Y_\varphi)$ makes equivalent to
\begin{equation*}
\forall X\subseteq\mathbb N\,\big(\cl_\varphi(X)\to x\in X\big).
\end{equation*}
It remains to show $x\in X$ for an arbitrary $X\subseteq\mathbb N$ with $\cl_\varphi(X)$. Again by $\lf_\varphi(Y_\varphi)$, we see that any such~$X$ validates
\begin{equation*}
\forall y\in\mathbb N\,(y\in Y_\varphi\to y\in X).
\end{equation*}
Given that all occurrences of~$X$ in the operator form~$\varphi$ are positive, we get
\begin{equation*}
\varphi(x,Y_\varphi)\to\varphi(x,X).
\end{equation*}
Due to the assumptions $\varphi(x,Y_\varphi)$ and $\cl_\varphi(X)$, this yields $\varphi(x,X)$ and then $x\in X$, as required. For the axiom~(L) as above, we note that (L)$^\star$ is given by
\begin{equation*}
\forall x\in\mathbb N\,\big(\varphi(x,\psi^\star)\to\psi^\star(x)\big)\to\forall x\in\mathbb N\,\big(x\in Y_\varphi\to\psi^\star(x)\big),
\end{equation*}
as the operator form $\varphi$ contains no predicate symbols $I_\theta$. Given that $\psi^\star$ is arith\-metical, we may form~$X\subseteq\mathbb N$ with
\begin{equation*}
\forall y\in\mathbb N\,\big(x\in X\leftrightarrow\psi^\star(y)\big).
\end{equation*}
This makes (L)$^\star$ equivalent to
\begin{equation*}
\cl_\varphi(X)\to\forall x\in\mathbb N\,(x\in Y_\varphi\to x\in X),
\end{equation*}
which is an immediate consequence of $\lf_\varphi(Y_\varphi)$.
\end{proof}

The following important result (or rather the equivalence with $\neg\varphi(x)$ in the proof) is sometimes called `Kleene normal form theorem'. To connect with the hom\-on\-y\-mous result from computability theory, one may imagine that $\theta_0$ describes a computation that can only use finitely many values of the oracle~$f$. The result does also show why ordinal analysis has a broader scope then it may first appear: by analyzing well orders, one can characterize all $\Pi^1_1$-formulas that are provable in an axiom system.

\begin{proposition}\label{prop:Kleene-normal-form}
For each $\Pi^1_1$-formula~$\varphi(x)$ there is an arithmetically definable (in fact recursive) family of binary relations $\tl_x$ on $\mathbb N$ such that we have
\begin{equation*}
\aca_0\vdash\forall x\in\mathbb N\,\big(\varphi(x)\leftrightarrow\text{``\,$\tl_x$ is well founded"}\,\big).
\end{equation*}
The definition of~$\tl_x$ has the same parameters as~$\varphi$.
\end{proposition}
\begin{proof}[Proof sketch]
We present the main ideas but refer to~\cite[Lemma~V.1.4]{simpson09} for details. Write $\neg\varphi(x)=\exists X\subseteq\mathbb N.\,\theta(x,X)$ with an arithmetical formula~$\theta$. For a function $f:\mathbb N\to\mathbb N$, let $f[m]$ denote the finite sequence $\langle f(0),\ldots,f(m-1)\rangle$. Similarly, we write $\sigma[m]:=\langle\sigma_0,\ldots,\sigma_{m-1}\rangle$ for the initial segment of a finite $\sigma=\langle\sigma_0,\ldots,\sigma_{n-1}\rangle$ with~$m\leq n$. To explain $X[m]$ for $X\subseteq\mathbb N$, we identify the latter with its characteristic function. By introducing Skolem functions, we find a decidable property~$\theta_0$ such that $\theta(x,X)$ holds precisely when there is an $f:\mathbb N\to\mathbb N$ with $\theta_0(x,X[m],f[m])$ for all~$m\in\mathbb N$. In particular, we may assume that $\theta_0$ is given by an arithmetical formula. The existentially quantified $X$ in $\neg\varphi$ can also be coded into~$f$, so that we get a modified~$\theta_0$ with
\begin{equation*}
\aca_0\vdash\neg\varphi(x)\leftrightarrow\exists f:\mathbb N\to\mathbb N\,\forall m\in\mathbb N.\,\theta_0(x,f[m]).
\end{equation*}
Let us now declare that $\sigma\tl_x\tau$ holds if $\sigma$ codes a sequence $\langle\sigma_0,\ldots,\sigma_{n-1}\rangle$ with $\theta_0(x,\sigma[m])$ for all $m\leq n$ and if $\tau=\sigma[k]$ with $k<n$ is a proper initial segment. When $f$ witnesses~$\neg\varphi(x)$, we have $f[m+1]\tl_x f[m]$ for all~$m$, so that $\tl_x$ is not well founded. Conversely, if we have an infinitely descending chain $\sigma^0,\sigma^1,\ldots$ with respect to $\tl_x$, then we can witness $\neg\varphi(x)$ by stipulating $f[m]:=\sigma^m[m]$. One often thinks of~$f$ as a branch in the tree of all finite descending chains for~$\tl_x$.
\end{proof}

Finally, we derive the promised conservativity result:

\begin{theorem}\label{thm:pica-id-conserv}
For each $\lid$-formula $\psi$ we have
\begin{equation*}
\pica^- + \{\lf_\varphi(Y_\varphi)\,|\,\varphi\text{ is an operator form}\}\vdash\psi^\star\quad\Rightarrow\quad\id_1\vdash\psi.
\end{equation*}
\end{theorem}
\begin{proof}
Given an arbitrary model $\mathcal N$ of $\id_1$, let $\mathcal S$ be the collection of definable sets, i.\,e., of sets that have the form $\{x\in N\,|\,\mathcal N\vDash\psi(x)\}$ for some $\lid$-formula~$\psi$, which may have parameters from the domain~$N$ of~$\mathcal N$. In the terminology of second order arithmetic, these are the sets that are arithmetically definable relative to the interpretations $I_\varphi^{\mathcal N}\subseteq N$ of the fixed point predicates. Let us also write $\mathcal N$ for the restriction of the given model to the language~$\lpa$. Below we will show
\begin{equation*}
(\mathcal N,\mathcal S)\vDash\pica^- + \left\{\lf_\varphi(I_\varphi^{\mathcal N})\,|\,\varphi\text{ is an operator form}\right\},
\end{equation*}
where $I_\varphi^{\mathcal N}\in\mathcal S$ serves as the interpretation of the set variable $Y_\varphi$. We thus get
\begin{equation*}
\id_1\nvDash\psi\quad\Rightarrow\quad\pica^- + \{\lf_\varphi(Y_\varphi)\,|\,\varphi\text{ is an operator form}\}\nvDash\psi^\star,
\end{equation*}
so that the theorem follows by completeness and soundness, analogous to the proof of Proposition~6.5 from the first part of these lecture notes~\cite{first-course}. As in the latter, we see that $(\mathcal N,\mathcal S)$ is a model of~$\aca_0$. In view of the definition of~$\mathcal S$, the claim that $(\mathcal N,\mathcal S)$ validates $\lf_\varphi(I_\varphi^{\mathcal N})$ amounts to
\begin{equation*}
I_\varphi^{\mathcal N}=\{x\in N\,|\,\mathcal N\vDash\psi(x)\text{ for all $\lid$-formulas~$\psi$ with }\mathcal N\vDash\cl_\varphi(\psi)\}.
\end{equation*}
Here $\subseteq$ holds because $\mathcal N$ satisfies axiom~(L). For the converse inclusion we note that (F) yields $\mathcal N\vDash\cl_\varphi(I_\varphi)$. The last and crucial task is to show that $(\mathcal N,\mathcal S)$ validates parameter-free $\Pi^1_1$-comprehension. Due to Proposition~\ref{prop:Kleene-normal-form}, it suffices to establish
\begin{equation*}
\{y\in N\,|\,(\mathcal N,\mathcal S)\vDash\text{``\,$\tl_y$ is well founded"}\}\in\mathcal S,
\end{equation*}
where the relations $\tl_y$ may be defined by any $\lpa$-formula. Consider the operator form~$\varphi(x,X)$ which asserts that $x$ codes a pair $\langle y,n\rangle$ such that we have $X\langle y,m\rangle$ for all $m$ that satisfy $m\tl_y n$ (note the connection with part~(c) of Exercise~\ref{ex:ind-def}). In view of the definition of~$\mathcal S$, we need only show
\begin{equation*}
(\mathcal N,\mathcal S)\vDash\text{``\,$\tl_y$ is well founded"}\quad\Leftrightarrow\quad\mathcal N\vDash\forall n\in\mathbb N.\,I_\varphi\langle y,n\rangle.
\end{equation*}
To avoid the unbounded quantifier on the right, one could replace $\forall n\in\mathbb N.\,I_\varphi\langle y,n\rangle$ by $I_\varphi\langle y,\langle\rangle\rangle$, noting that the empty sequence is the largest element of the order~$\tl_y$ that was constructed in the previous proof. In order to establish the right side by contradiction, we assume that the set
\begin{equation*}
\{n\in N\,|\,\mathcal N\vDash \neg I_\varphi\langle y,n\rangle\}\in\mathcal S
\end{equation*}
is nonempty. Given the left side of our equivalence, we get an $n\in N$ such that $\mathcal N$ validates $\neg I_\varphi\langle y,n\rangle$ and $m\tl_y n\to I_\varphi\langle y,m\rangle$. The latter amounts to~$\varphi(\langle y,n\rangle,I_\varphi)$. We thus get the contradictory $I_\varphi\langle y,n\rangle$ by axiom~(F) in~$\mathcal N$. For the direction from right to left, we consider an arbitrary set $\{n\in N\,|\,\mathcal N\vDash\psi(n)\}\in\mathcal S$ that does not have a $\tl_y$-minimal element. We must derive that this set is empty. Let $\psi'(x)$ be the statement that $x$ codes a pair $\langle y,n\rangle$ with $\neg\psi(n)$. We shall prove $\mathcal N\vDash\cl_\varphi(\psi')$. Given the latter, we get $I_\varphi x\to\psi'(x)$ by axiom~(L) in~$\mathcal N$. Assuming the right side of our equivalence, this yields $\psi'(\langle y,n\rangle)$ and hence~$\neg\psi(n)$ for all~$n$, as required. To establish $\cl_\varphi(\psi')$ in~$\mathcal N$, we derive $\psi'(x)$ from $\varphi(x,\psi')$. The latter asserts that $x$ codes a pair $\langle y,n\rangle$ such that $m\tl_y n$ entails $\psi'(\langle y,m\rangle)$ and hence $\neg\psi(m)$. Due to the assumption that $\{n\in N\,|\,\mathcal N\vDash\psi(n)\}$ has no $\tl_y$-minimal element, we must have~$\neg\psi(n)$ and hence $\psi'(x)$, as required.
\end{proof}

\section{Infinite proofs and operator control}\label{sect:operator-control}

In the present section, we introduce a system of infinite proofs that is suitable for our ordinal analysis of~$\id_1$. This system involves a notion of operator control that is due to Wilfried Buchholz~\cite{buchholz-local-predicativity}. The latter ensures that certain uncountable proofs can be collapsed into countable ones, as we shall see in the subsequent sections.

Let us recall the well order $\vartheta(\varepsilon_{\Omega+1})$ from Section~\ref{sect:Bachmann-Howard}. We define $\lido$ as the extension of $\lpa$ by a unary predicate symbol $I_\varphi^{\prec \alpha}$ for each operator form~$\varphi$ and each term $\alpha\in\bh$ with $\alpha\preceq\Omega$. The intended interpretation is suggested by part~(d) of Exercise~\ref{ex:ind-def}. To view $\lid$ as a sublanguage of~$\lido$, we identify the predicate symbols $I_\varphi$ and $I_\varphi^{\prec\Omega}$. The aforementioned exercise motivates this identification if we think of~$\Omega$ as the first uncountable ordinal (cf.~the explanation after Definition~\ref{def:Bachmann-Howard}). It also explains the restriction to $\alpha\preceq\Omega$ in the definition of $\lido$.

To capture the intended interpretation of $I_\varphi^{\prec\alpha}$ on a syntactic level, we shall now define an assignment of infinite disjunctions and conjunctions. In the previous lecture~\cite{first-course}, we have encountered this approach in a simpler setting. As preparation, we embed~$\mathbb N$ into $\vartheta(\varepsilon_{\Omega+1})$ by setting $n:=\langle 0,\ldots,0\rangle$ with $n$~entries~$0=\langle\rangle$. The range of this embedding is the initial segment below~$\omega:=\langle 1\rangle$. We will also write $n$ for the $n$-th numeral. All formulas are assumed to be in negation normal form, as explained in the paragraph after Exercise~\ref{ex:ind-def}. By a false literal we mean one that is false in the standard model of the natural numbers.

\begin{definition}\label{def:disj-conj}
To each $\lido$-sentence $\psi$ we assign either a disjunction $\psi\simeq\bigvee_{\gamma\prec\alpha}\psi_\gamma$ or a conjunction $\psi\simeq\bigwedge_{\gamma\prec\alpha}\psi_\gamma$ with $\alpha\preceq\Omega$ in $\vartheta(\varepsilon_{\Omega+1})$. The disjunctions are
\begin{gather*}
\psi\simeq\text{``the empty disjunction"}\quad\text{when $\psi$ is a false literal of~$\lpa$},\\
I_\varphi^{\prec\alpha}t\simeq\textstyle\bigvee_{\gamma\prec\alpha}\varphi(t,I_\varphi^{\prec\gamma}),\quad\psi_0\lor\psi_1\simeq\textstyle\bigvee_{i\prec 2}\psi_i,\quad\exists x\in\mathbb N.\,\psi(x)\simeq\textstyle\bigvee_{n\prec\omega}\psi(n).
\end{gather*}
The conjunctions are determined by $\neg\psi\simeq\bigwedge_{\gamma\prec\alpha}\neg\psi_\gamma$ for $\psi\simeq\bigvee_{\gamma\prec\alpha}\psi_\gamma$.
\end{definition}

To confirm that each $\lido$-sentence is either disjunctive or conjunctive, one should recall that $\neg\neg\varphi$ and $\varphi$ denote the same formula. For the same reason, one can conclude that $\neg\psi\simeq\bigvee_{\gamma\prec\alpha}\neg\psi_\gamma$ follows from $\psi\simeq\bigwedge_{\gamma\prec\alpha}\psi_\gamma$. Hence $(\neg\psi)_\gamma$ is always given by $\neg(\psi_\gamma)$, so that the parentheses may be omitted. To become familiar with the notation, one may wish to write out the conjunctive cases explicitly.

Our assignment of disjunctions and conjunctions amounts to an inductive definition of truth for $\lido$-sentences or, in other words, to an infinitary proof system in which $\psi\simeq\bigwedge_{\gamma\prec\alpha}\psi_\gamma$ can be deduced with premises $\psi_\gamma$ for all $\gamma\prec\alpha$. Of course, the term system $\vartheta(\varepsilon_{\Omega+1})$ is countable. At the same time, we have mentioned that the first uncountable ordinal provides a reasonable interpretation for $\Omega$. So at least intuitively, we have rules with uncountably many premises, as indicated above.

If the inductive definition of truth is to be well founded, then each $\psi_\gamma$ must be less complex than $\psi\simeq\bigwedge_{\gamma\prec\alpha}\psi_\gamma$ in some sense. We shall ensure this by a suitable definition of formula rank. This relies on basic operations of ordinal arithmetic. We declare that addition on~$\vartheta(\varepsilon_{\Omega+1})$ is given by the clause from Definition~5.3 of the previous lecture~\cite{first-course}, where we identify $\Omega$ and $\vartheta \alpha$ with $\langle\Omega\rangle$ and $\langle\vartheta\alpha\rangle$, respectively. Multiplication with left factor~$\omega$ is defined by the clauses
\begin{equation*}
\omega\cdot\Omega:=\Omega,\quad\omega\cdot\vartheta\alpha:=\vartheta\alpha,\quad\omega\cdot\langle\alpha_0,\ldots,\alpha_{n-1}\rangle:=\langle1+\alpha_0,\ldots,1+\alpha_{n-1}\rangle.
\end{equation*}
These are motivated by the set theoretic equality
\begin{equation*}
\omega^1\cdot\left(\omega^{\alpha_0}+\ldots+\omega^{\alpha_{n-1}}\right)=\omega^{1+\alpha_0}+\ldots+\omega^{1+\alpha_{n-1}}
\end{equation*}
and the intuition that $\Omega$ and $\vartheta\alpha$ represent $\varepsilon$-numbers. For later use, we also define a function $\omega:\vartheta(\varepsilon_{\Omega+1})\to\vartheta(\varepsilon_{\Omega+1})$ that represents exponentiation with base~$\omega$ and is given by
\begin{equation*}
\omega(\alpha):=\begin{cases}
\alpha & \text{if $\alpha$ has the form $\Omega$ or $\vartheta\alpha'$},\\
\langle\alpha\rangle & \text{otherwise}.
\end{cases}
\end{equation*}
Basic laws of ordinal arithmetic can be verified on a syntactic level. In particular, it is not hard to see that the map $\alpha\mapsto\omega\cdot\alpha$ is order preserving and that $\gamma\prec\omega\cdot\alpha$ entails $\gamma+1\prec\omega\cdot\alpha$. Let us also adopt the results from Exercise~5.4 of the first lecture course~\cite{first-course}. We trust that the reader will identify and check similar basic facts that are used in the sequel.

\begin{definition}
To each $\lido$-formula $\psi$ we assign a rank $\rk(\psi)\in\vartheta(\varepsilon_{\Omega+1})$ by
\begin{gather*}
\rk(\psi):=0\quad\text{when $\psi$ is a literal of $\lpa$},\\
\rk(I_\varphi^{\prec\alpha}t):=\rk(\neg I_\varphi^{\prec\alpha}t):=\omega\cdot\alpha,\\
\rk(\psi_0\lor\psi_1):=\rk(\psi_0\land\psi_1):=\max\{\rk(\psi_0),\rk(\psi_1)\}+1,\\
\rk(\exists x\in\mathbb N.\,\psi):=\rk(\forall x\in\mathbb N.\,\psi):=\rk(\psi)+1.
\end{gather*}
\end{definition}

Note that we always have $\rk(\psi)=\rk(\neg\psi)$, which preserves the duality between disjunctive and conjunctive formulas. For the following exercise, this means that only the disjunctive cases need to be considered explicitly.

\begin{exercise}\label{ex:ranks}
For $\psi\simeq\bigvee_{\gamma\prec\alpha}\psi_\gamma$ or $\psi\simeq\bigwedge_{\gamma\prec\alpha}\psi_\gamma$, prove $\rk(\psi_\gamma)\prec\rk(\psi)$ for $\gamma\prec\alpha$.
\end{exercise}

In order to collapse uncountable proofs into countable ones, we will use the function $\vartheta:\vartheta(\varepsilon_{\Omega+1})\to\vartheta(\varepsilon_{\Omega+1})\cap\Omega$ from Section~\ref{sect:Bachmann-Howard}. We have seen that the latter preserves the order between many but not all pairs of arguments. For this reason, it will be important that only certain ordinal terms occur in our proofs. To ensure this, we use a method of operator control that is due to Wilfried Buchholz~\cite{buchholz-local-predicativity}.

\begin{definition}\label{def:nice-operator}
Write $\mathcal P$ for the powerset of $\vartheta(\varepsilon_{\Omega+1})$. By a closure operator (short: operator), we mean a function $\mathcal H:\mathcal P\to\mathcal P$ such that all $X,Y\in\mathcal P$ validate
\begin{equation*}
X\subseteq\mathcal H(X)\qquad\text{and}\qquad X\subseteq\mathcal H(Y)\,\Rightarrow\,\mathcal H(X)\subseteq\mathcal H(Y).
\end{equation*}
We say that $\mathcal H$ is nice if
\begin{equation*}
\alpha\in\mathcal H(X)\quad\Leftrightarrow\quad E(\alpha)\subseteq\mathcal H(X)
\end{equation*}
holds for all~$\alpha\in\vartheta(\varepsilon_{\Omega+1})$ and $X\in\mathcal P$ (see Definition~\ref{def:Bachmann-Howard} for information on~$E$). For an operator $\mathcal H$ and $Z\in\mathcal P$, we define $\mathcal H[Z]:\mathcal P\to\mathcal P$ by $\mathcal H[Z](X):=\mathcal H(Z\cup X)$.
\end{definition}

We declare that the ordinal parameters of an $\lido$-formula~$\psi$ are given by
\begin{equation*}
k(\psi):=\{\alpha\in\vartheta(\varepsilon_{\Omega+1})\,|\,\psi\text{ contains a literal }I_\varphi^{\prec\alpha}t\text{ or }\neg I_\varphi^{\prec\alpha}t\}.
\end{equation*}
Parts of the following exercise are taken from Lemma~3.6 of~\cite{buchholz-local-predicativity}.

\begin{exercise}\label{ex:operators}
(a) Check that $X\mapsto\mathcal H_\gamma(X)$ from Definition~\ref{def:H_gamma} is a nice operator.

(b) Show that $X\subseteq Y$ implies $\mathcal H(X)\subseteq\mathcal H(Y)$ for any operator~$\mathcal H$.

(c) Prove that $\mathcal H[Z]$ is a (nice) operator when the same holds for~$\mathcal H$. Show that $\mathcal H[Z]$ and $\mathcal H$ coincide when we have $Z\subseteq\mathcal H(\emptyset)$. Also note $\mathcal H[Z][Z']=\mathcal H[Z\cup Z']$.

(d) Observe that an operator $\mathcal H$ is nice precisely when all~$X$ validate
\begin{equation*}
\Omega\in\mathcal H(X)\quad\text{and}\quad\langle\alpha_0,\ldots,\alpha_{n-1}\rangle\in\mathcal H(X)\,\Leftrightarrow\,\{\alpha_0,\ldots,\alpha_{n-1}\}\subseteq\mathcal H(X).
\end{equation*}

(e) Let us assume that~$\mathcal H$ is a nice operator. Given $\alpha,\beta\in\mathcal H(X)$, show that $\mathcal H(X)$ contains~$\alpha+\beta$ as well as $\omega\cdot\alpha$ and $\omega(\alpha)$. Then derive that we have $\rk(\psi)\in\mathcal H(k(\psi))$ for any $\lido$-formula~$\psi$.
\end{exercise}

In the context of infinite proofs, we shall always assume that formulas are closed, unless noted otherwise. This makes sense because $\forall x\in\mathbb N.\,\psi(x)\simeq\bigwedge_{n\prec\omega}\psi(n)$ can be derived from the infinitely many premises $\psi(n)$ without the use of a free variable. By an $\lido$-sequent we shall mean a finite set of $\lido$-sentences. In the context of sequents, it is common to omit parentheses $\{\cdot\}$ and to denote unions by commata. Specifically, we write $\Gamma,\Delta$ and $\Gamma,\psi$ at the place of~$\Gamma\cup\Delta$ and $\Gamma\cup\{\psi\}$, respectively, where $\Gamma$ and $\Delta$ are sequents while $\psi$ is a formula. As usual, the sequent $\Gamma=\psi_0,\ldots,\psi_{n-1}=\{\psi_0,\ldots,\psi_{n-1}\}$ represents the disjunction $\bigvee\Gamma=\psi_0\lor\ldots\lor\psi_{n-1}$. Its parameters are given by
\begin{equation*}
k(\Gamma):=\textstyle\bigcup_{i<n}k(\psi_i)\quad\text{for}\quad\Gamma=\psi_0,\ldots,\psi_{n-1}.
\end{equation*}
We can now specify our infinitary proof system, which is further explained below. The following definition is inspired by Theorem~3.8 of~\cite{buchholz-local-predicativity}.

\begin{definition}\label{def:H-controlled-proofs}
By recursion on~$\alpha\in\vartheta(\varepsilon_{\Omega+1})$, we declare that the relation $\mathcal H\vdash^\alpha_\rho\Gamma$ between a nice operator~$\mathcal H$, an element $\rho\in\vartheta(\varepsilon_{\Omega+1})$ and an $\lido$-sequent~$\Gamma$ holds precisely when we have $\{\alpha\}\cup k(\Gamma)\subseteq\mathcal H(\emptyset)$ and one of the following clauses applies:
\begin{align*}
\left(\textstyle\bigwedge\right)\quad&\parbox[t]{.75\textwidth}{for some conjunctive sentence $\psi\simeq\bigwedge_{\gamma\prec\delta}\psi_\gamma\in\Gamma$ and all $\gamma\prec\delta$, we have $\mathcal H[\{\gamma\}]\vdash^{\alpha(\gamma)}_\rho\Gamma,\psi_\gamma$ for some $\alpha(\gamma)\prec\alpha$,}\\
\left(\textstyle\bigvee\right)\quad&\parbox[t]{.75\textwidth}{for some disjunctive $\psi\simeq\bigvee_{\gamma\prec\delta}\psi_\gamma\in\Gamma$ and some $\gamma\prec\delta$ with $\gamma\prec\alpha$ and $\gamma\in\mathcal H(\emptyset)$, we have $\mathcal H\vdash^{\alpha'}_\rho\Gamma,\psi_\gamma$ for some~$\alpha'\prec\alpha$,}\\
(\mathsf{Cut})\quad&\parbox[t]{.75\textwidth}{for some $\lido$-sentence $\psi$ with $\rk(\psi)\prec\rho$, we have $\mathcal H\vdash^{\alpha'}_\rho\Gamma,\psi$ as well as $\mathcal H\vdash^{\alpha'}_\rho\Gamma,\neg\psi$ for some $\alpha'\prec\alpha$,}\\
(\mathsf{Fix})\quad&\parbox[t]{.75\textwidth}{we have $\Omega\preceq\alpha$ and there is a sentence $I_\varphi^{\prec\Omega}t\in\Gamma$ such that we have $\mathcal H\vdash^{\alpha'}_\rho\Gamma,\varphi(t,I_\varphi^{\prec\Omega})$ for some~$\alpha'\prec\alpha$.}
\end{align*}
\end{definition}

We often write $\mathcal H[\gamma_0,\ldots,\gamma_{n-1}]$ at the place of $\mathcal H[\{\gamma_0,\ldots,\gamma_{n-1}\}]$. In the following section, we will see that $\mathcal H[k(\psi)]\vdash^{\omega\cdot\rk(\psi)}_0\psi,\neg\psi$ holds for any $\lido$-sentence~$\psi$. The exercise below shows that $\mathcal H\vdash^\alpha_\rho\Gamma$ weakens to $\mathcal H[k(\Delta)]\vdash^\alpha_\rho\Gamma,\Delta$. To give an example of a derivation, we assume these facts and derive $\neg I^{\prec\alpha}_\varphi t,I^{\prec\beta}_\varphi t$ for~$\alpha\preceq\beta\preceq\Omega$ (which intuitively yields $I^{\prec\alpha}_\varphi\subseteq I^{\prec\beta}_\varphi$), where we abbreviate $\delta(\gamma):=\max\{\gamma,\rk(\varphi(t,I_\varphi^{\prec\gamma}))\}$:
\begin{equation*}
\AxiomC{$\mathcal H[\alpha,\beta,\gamma]\sststile{0}{\omega\cdot\delta(\gamma)}\neg I^{\prec\alpha}_\varphi t,I^{\prec\beta}_\varphi t,\neg\varphi(t,I_\varphi^{\prec\gamma}),\varphi(t,I_\varphi^{\prec\gamma})$}
\RightLabel{($\bigvee$)}
\UnaryInfC{$\cdots\,\,\mathcal H[\alpha,\beta,\gamma]\sststile{0}{\omega\cdot\delta(\gamma)+1}\neg I^{\prec\alpha}_\varphi t,\neg\varphi(t,I_\varphi^{\prec\gamma}),I_\varphi^{\prec\beta}t\,\,\cdots\!\!\!\!\!$}
\AxiomC{(all $\gamma\prec\alpha$)}
\RightLabel{($\bigwedge$).}
\BinaryInfC{$\mathcal H[\alpha,\beta]\sststile{0}{\omega^2\cdot\alpha}\neg I^{\prec\alpha}_\varphi t,I_\varphi^{\prec\beta}$}
\DisplayProof
\end{equation*}
The example illustrates why the premise of clause~($\bigwedge$) involves $\mathcal H[\gamma]$ rather than~$\mathcal H$. In order to provide further intuition, we point out that infinite derivations can be seen as labelled trees. Each node of such a trees can be identified with a finite sequence $\langle\gamma_0,\ldots,\gamma_{n-1}\rangle$ for $\gamma_i\prec\Omega$. Here the $\gamma_i$ indicate the premises that lead from the root to that node, with $\gamma_i=0$ or $\gamma_i=1$ for rules other than~($\bigwedge$). For~example, the top node in the derivation above corresponds to the sequence~$\langle\gamma,0\rangle$. If $\alpha$ and~$\Gamma$ are the ordinal and sequent label at the node $\langle\gamma_0,\ldots,\gamma_{n-1}\rangle$, then the condition from Definition~\ref{def:H-controlled-proofs} amounts to
\begin{equation*}
\{\alpha\}\cup k(\Gamma)\subseteq\mathcal H[\gamma_0,\ldots,\gamma_{n-1}](\emptyset)=\mathcal H(\{\gamma_0,\ldots,\gamma_{n-1}\}).
\end{equation*}
In this sense, the operator~$\mathcal H$ controls the parameters in the entire derivation. Some informal motivation for operator control has been given at the end of Section~\ref{sect:Bachmann-Howard}, but only the proof of Theorem~\ref{thm:collapsing} will explain all technical details of our approach. To become familiar with the formal side of Definition~\ref{def:H-controlled-proofs}, it may help to verify the aforementioned weakening result in detail (cf.~\cite[Lemma~3.9]{buchholz-local-predicativity}):

\begin{exercise}\label{ex:weakening}
(a) Prove that $\mathcal H\vdash^\alpha_\rho\Gamma$ entails $\mathcal H'\vdash^\alpha_\rho\Gamma$ when we have $\mathcal H(X)\subseteq\mathcal H'(X)$ for all subsets~$X$ of $\vartheta(\varepsilon_{\Omega+1})$.

(b) Show that $\mathcal H\vdash^\alpha_\rho\Gamma$ entails $\mathcal H\vdash^{\alpha'}_{\rho'}\Gamma,\Delta$ whenever we have $\alpha\preceq\alpha'$ and $\rho\preceq\rho'$ as well as $\{\alpha'\}\cup k(\Delta)\subseteq\mathcal H(\emptyset)$. \emph{Remark:} When the last condition fails, one can use part~(a) to replace~$\mathcal H$ by a suitable~$\mathcal H'$, such as $\mathcal H[\{\alpha'\}\cup k(\Delta)]$.
\end{exercise}

Some readers may have wondered why both premises of~($\mathsf{Cut}$) should be derived with the same ordinal height~$\alpha'$. With weakening at hand, we see that this is inessential, because we can match two heights by increasing the smaller one. In the rest of this section, we discuss other aspects of Definition~\ref{def:H-controlled-proofs} and draw some important consequences. First, the condition $\Omega\preceq\alpha$ in clause~($\mathsf{Fix}$) is a somewhat ad hoc but highly effective way to obtain the following soundness result for derivations that have been collapsed to height below~$\Omega$ (cf.~Exercise~\ref{ex:ind-def}).

\begin{proposition}\label{prop:soundness-id}
Extend the standard structure for $\lpa$ into an $\lido$-structure~$\mathcal N$ by recursively declaring
\begin{equation*}
I^{\prec\delta,\,\mathcal N}_\varphi:=\textstyle\bigcup_{\gamma\prec\delta}I^{\gamma,\,\mathcal N}_\varphi\quad\text{with}\quad I^{\gamma,\,\mathcal N}_\varphi:=\{n\in\mathbb N\,|\,\mathcal N\vDash\varphi(n,I^{\prec\gamma,\,\mathcal N}_\varphi)\}.
\end{equation*}
Given $\mathcal H\vdash^\alpha_\rho\Gamma$ with $\alpha\prec\Omega$, we then get $\mathcal N\vDash\bigvee\Gamma$.
\end{proposition}
\begin{proof}
One argues by induction on~$\alpha$ and distinguishes cases according to the clauses from Definition~\ref{def:H-controlled-proofs}. Clause~($\mathsf{Fix}$) cannot apply, due to the assumption~\mbox{$\alpha\prec\Omega$}. In the case of~($\mathsf{Cut}$), the premises yield $\mathcal N\vDash\psi\lor\bigvee\Gamma$ as well as $\mathcal N\vDash\neg\psi\lor\bigvee\Gamma$. Since we always have $\mathcal N\nvDash\psi$ or $\mathcal N\nvDash\neg\psi$, we can conclude $\mathcal N\vDash\bigvee\Gamma$, as desired. The cases of~($\bigwedge$) and~($\bigvee$) are covered by the following fact, which is straightforward by the definition of~$\mathcal N$: We have $\mathcal N\vDash\psi$ with $\psi\simeq\bigwedge_{\gamma\prec\delta}\psi_\gamma$ or $\psi\simeq\bigvee_{\gamma\prec\delta}\psi_\gamma$, respectively, precisely if $\mathcal N\vDash\psi_\gamma$ holds for all or some~$\gamma\prec\delta$.
\end{proof}

In clause~($\bigvee$), some readers may have wondered about the conditions $\gamma\prec\alpha$ and $\gamma\in\mathcal H(\emptyset)$. The latter is automatic when we have~$\gamma\in k(\psi_\gamma)$, as the initial condition from Definition~\ref{def:H-controlled-proofs} requires $k(\Gamma,\psi_\gamma)\subseteq\mathcal H(\emptyset)$. Now $\gamma\in k(\psi_\gamma)$ is satisfied in typical but not in all cases (note that $\varphi(t,I^{\prec\gamma}_\varphi)$ may not actually contain $I^{\prec\gamma}_\varphi$). The more substantial condition $\gamma\prec\alpha$ ensures that the ordinal label controls the size of existential witnesses, which allows us to extract quantitative content. This is made precise by the following result (cf.~\cite[Lemma~3.17]{buchholz-local-predicativity}), in which $\psi^\beta$ with $\beta\prec\Omega$ denotes the formula that results from~$\psi$ when all literals $I^{\prec\Omega}_\varphi t$ and $\neg I^{\prec\Omega}_\varphi t$ are replaced by  $I^{\prec\beta}_\varphi t$ and $\neg I^{\prec\beta}_\varphi t$, respectively.

\begin{theorem}[`Boundedness']\label{thm:boundedness}
Given $\mathcal H\vdash^\alpha_\rho\Gamma,\psi$, we obtain $\mathcal H\vdash^\alpha_\rho\Gamma,\psi^\beta$ for any element $\beta\in\mathcal H(\emptyset)$ with $\alpha\preceq\beta\prec\Omega$.
\end{theorem}
In the important case where we have~$\alpha=\beta$, the condition $\beta\in\mathcal H(\emptyset)$ is already ensured by $\mathcal H\vdash^\alpha_\rho\Gamma,\psi$, due to the initial condition from Definition~\ref{def:H-controlled-proofs}.
\begin{proof}
Once again, we argue by induction on~$\alpha$. In view of $k(\psi^\beta)\subseteq k(\psi)\cup\{\beta\}$, the assumption $\beta\in\mathcal H(\emptyset)$ ensures that the initial condition from Definition~\ref{def:H-controlled-proofs} is satisfied. As before, we now consider the different clauses. In the crucial case,~the formula from the proposition is of the form $\psi\simeq\bigvee_{\gamma\prec\delta}\psi_\gamma$ and has been introduced by clause~($\bigvee$). The premise of this clause will then provide
\begin{equation*}
\mathcal H\vdash^{\alpha'}_\rho \Gamma,\psi,\psi_\gamma
\end{equation*}
with $\alpha'\prec\alpha$ for some $\gamma\prec\delta$ with $\gamma\prec\alpha$ and $\gamma\in\mathcal H(\emptyset)$. We can derive
\begin{equation*}
\mathcal H\vdash^{\alpha'}_\rho \Gamma,\psi^\beta,(\psi_\gamma)^\beta
\end{equation*}
by two applications of the induction hypothesis. Let us now observe that we get $\psi^\beta\simeq\bigvee_{\gamma\prec\eta}(\psi_\gamma)^\beta$ with $\eta\in\{\beta,\delta\}$, which means that $(\psi^\beta)_\gamma=(\psi_\gamma)^\beta$ holds for $\gamma\prec\eta$. In case we have $\psi=I^{\prec\Omega}_\varphi t$, this is true since $\psi_\gamma=\varphi(t,I^{\prec\gamma}_\varphi)$ with $\gamma\prec\Omega$ does not involve~$I^{\prec\Omega}_\varphi$, so that we have~$(\psi_\gamma)^\beta=\psi_\gamma$. The other cases from Definition~\ref{def:disj-conj} are readily checked. Crucially, the condition $\gamma\prec\alpha$ ensures that the $\gamma$ from our premise satisfies~$\gamma\prec\eta$. We can thus reapply~($\bigvee$) to get $\mathcal H\vdash^\alpha_\rho\Gamma,\psi^\beta$. Let us also mention the case where the formula from the proposition has the form~$\psi\simeq\bigwedge_{\gamma\prec\delta}\psi_\gamma$ and is introduced by clause~($\bigwedge$). Here we get $\psi^\beta\simeq\bigwedge_{\gamma\prec\eta}(\psi_\gamma)^\beta$ with $\eta\in\{\beta,\delta\}$. In fact, we have either $\eta=\delta$ or $\eta=\beta\prec\Omega=\delta$, so that $\eta\preceq\delta$ holds in any case. Due to the latter, the number of premises in clause~($\bigwedge$) becomes smaller. For this reason, one can reapply the clause after the induction hypothesis has been used. As in the previous proof, clause~($\mathsf{Fix}$) is excluded by the condition~$\alpha\prec\Omega$. The remaining cases concern ($\mathsf{Cut}$) and applications of ($\bigvee$) and~($\bigwedge$) to formulas in~$\Gamma$. These are readily reduced to the induction hypothesis, as $\psi$ plays no distinguished role.
\end{proof}

As mentioned above, our embedding of~$\lid$ into~$\lido$ will identify $I_\varphi$ with $I^{\prec\Omega}_\varphi$. The following result shows that our infinitary proof system tracks information on the closure ordinals of inductive definitions (i.\,e., on the least~$\alpha$ with $I^\alpha_\Phi=I^{<\alpha}_\Phi$, as considered in part~(d) of Exercise~\ref{ex:ind-def}).

\begin{corollary}\label{cor:ind-stages}
Given $\mathcal H\vdash^\alpha_\rho I^{\prec\Omega}_\varphi n$ with $\alpha\prec\Omega$, we get $n\in I_\varphi^{\prec\alpha,\mathcal N}$ (where the latter is defined as in Proposition~\ref{prop:soundness-id}).
\end{corollary}
\begin{proof}
By the boundedness theorem we obtain~$\mathcal H\vdash^\alpha_\rho I^{\prec\alpha}_\varphi n$. Now the claim follows by the preceding proposition on soundness.
\end{proof}

It is instructive to derive the following special case of the corollary. We will later use it to show that $\pica^-$ cannot prove the well foundedness of the Bachmann-Howard ordinal~$\vartheta(\varepsilon_{\Omega+1})$.

\begin{exercise}\label{ex:limit-wfness-id1}
For a suitable $\lpa$-definition of the order~$\prec$ on $\vartheta(\varepsilon_{\Omega+1})$, we consider the operator form
\begin{equation*}
\varphi(x,X):=\forall y\in\mathbb N\,(y\prec x\to Xy).
\end{equation*}
Assume $\mathcal H\vdash^\alpha_\rho I^{\prec\Omega}_\varphi\delta$ with $\alpha\prec\Omega$ for (the numeral that codes) a term $\delta\in\vartheta(\varepsilon_{\Omega+1})$. Show that we must have $\delta\prec\alpha$. \emph{Hint:} Use induction on~$\alpha$ to prove
\begin{equation*}
I_\varphi^{\prec\alpha,\,\mathcal N}\cap\vartheta(\varepsilon_{\Omega+1})=\{\delta\in\vartheta(\varepsilon_{\Omega+1})\,|\,\delta\prec\alpha\}.
\end{equation*}
\end{exercise}

The reader may have observed that the previous results hold in the presence of the cut rule, i.\,e., that we have never required~$\rho=0$ or even just~$\rho\prec\Omega$. In sharp contrast with the predicative case, it seems fair to say that cut elimination does only play a secondary role in the ordinal analysis of impredicative systems such as~$\id_1$ (even though we will need to reduce certain cuts in our proof of collapsing). Let us also point out that Corollary~\ref{cor:ind-stages} in itself is not a deep result -- it holds by the design of our infinitary proof system. To make it significant, we need an embedding of $\id_1$ and a collapsing procedure that yields proofs with height~$\alpha\prec\Omega$. These requirements will be satisfied in the following sections.

\section{Embedding and collapsing}\label{sect:embedding}

In the present section, we prove the embedding and collapsing results that were promised above. The following is a first ingredient for our embedding theorem. We shall always assume that $\mathcal H$ is a nice operator.

\begin{lemma}\label{lem:LEM}
We have $\mathcal H[k(\psi)]\vdash^{\omega\cdot\rk(\psi)}_0\psi,\neg\psi$ for any $\lido$-sentence~$\psi$.
\end{lemma}
\begin{proof}
We argue by induction on~$\rk(\psi)$. Due to the duality between conjunctive and disjunctive formulas, we may assume $\psi\simeq\bigwedge_{\gamma\prec\delta}\psi_\gamma$ and $\neg\psi\simeq\bigvee_{\gamma\prec\delta}\neg\psi_\gamma$. Note that $k(\psi_\gamma)\subseteq k(\psi)\cup\{\gamma\}$ holds in all cases from Definition~\ref{def:disj-conj}. Due to the weakening result from Exercise~\ref{ex:weakening}, the induction hypothesis will thus yield
\begin{equation*}
\mathcal H[k(\psi)][\gamma]\vdash^{\omega\cdot\rk(\psi_\gamma)}_0\psi,\neg\psi,\psi_\gamma,\neg\psi_\gamma
\end{equation*}
for each~$\gamma\prec\delta$. We can conclude
\begin{equation*}
\mathcal H[k(\psi)][\gamma]\vdash^{\alpha(\gamma)}_0\psi,\neg\psi,\psi_\gamma\quad\text{for}\quad\alpha(\gamma):=\max\{\omega\cdot\rk(\psi_\gamma),\gamma\}+1
\end{equation*}
by clause~($\bigvee$) of Definition~\ref{def:H-controlled-proofs}, as the maximum ensures $\gamma\prec\alpha(\gamma)$. It is not hard to check that we have $\gamma\prec\omega\cdot\rk(\psi)$ for any $\gamma\prec\delta$. Together with the facts that were proved in Exercises~\ref{ex:ranks} and~\ref{ex:operators}, this yields
\begin{equation*}
\alpha(\gamma)\prec\omega\cdot\rk(\psi)\in\mathcal H[k(\psi)](\emptyset).
\end{equation*}
We can thus conclude by clause ($\bigwedge$) of Definition~\ref{def:H-controlled-proofs}.
\end{proof}

As mentioned in the previous section, our embedding of~$\lid$ into~$\lido$ identifies the predicate symbols $I_\varphi$ and~$I_\varphi^{\prec\Omega}$. Due to the rule ($\mathsf{Fix}$) of our infinitary proof system, we can now derive the instances of axiom schema~(F) from Section~\ref{sect:ind-def}.

\begin{proposition}
For each operator form~$\varphi(x,X)$, there is an $n\prec\omega$ with
\begin{equation*}
\mathcal H\vdash^{\Omega+\omega\cdot n}_0\forall x\in\mathbb N\,\left(\varphi\left(x,I_\varphi^{\prec\Omega}\right)\to I_\varphi^{\prec\Omega}x\right).
\end{equation*}
\end{proposition}
\begin{proof}
For $\psi:=\varphi(x,I_\varphi^{\prec\Omega})$ we have $k(\psi)\subseteq\{\Omega\}$ and $\rk(\psi)\preceq\Omega+m$ for some~$m\in\mathbb N$, as $\varphi$ is an~$\lpax$-formula and we have $\omega\cdot\Omega=\Omega$. From Exercise~\ref{ex:operators} we know~$\mathcal H[\Omega]=\mathcal H$. For any~$k\prec\omega$, the previous lemma and weakening will thus yield
\begin{equation*}
\mathcal H\vdash^{\omega\cdot(\Omega+m)}_0\neg\varphi(k,I_\varphi^{\prec\Omega}),\varphi(k,I_\varphi^{\prec\Omega}),I_\varphi^{\prec\Omega} k.
\end{equation*}
We can now apply clause~($\mathsf{Fix}$) from Definition~\ref{def:H-controlled-proofs} to get
\begin{equation*}
\mathcal H\vdash^{\omega\cdot(\Omega+m)+1}_0\neg\varphi(k,I_\varphi^{\prec\Omega}),I_\varphi^{\prec\Omega} k.
\end{equation*}
Two applications of~($\bigvee$) transform $\neg\varphi(k,I_\varphi^{\prec\Omega}),I_\varphi^{\prec\Omega} k$ into $\neg\varphi(k,I_\varphi^{\prec\Omega})\lor I_\varphi^{\prec\Omega} k$, which denotes the same formula as $\varphi(k,I_\varphi^{\prec\Omega})\to I_\varphi^{\prec\Omega} k$, since we work with negation normal forms (cf.~the proof of Theorem~3.7 from~\cite{first-course}). Since~$k$ was arbitrary, we can use~($\bigwedge$) to conclude with a bound $\omega\cdot(\Omega+m)+i\prec\omega\cdot(\Omega+m+1)=\Omega+\omega\cdot(m+1)$.
\end{proof}

For the embedding of axiom schema~(L) from Section~\ref{sect:ind-def}, we will use the following monotonicity result, which coincides with Lemma~9.5.4 of~\cite{pohlers-proof-theory}.

\begin{exercise}
Consider $\omega\preceq\alpha$ and assume that $\mathcal H\vdash^\alpha_\rho\Gamma,\neg\theta(s),\psi(s)$ holds for every closed $\lpa$-term~$s$. Given an operator form~$\varphi$, show that there is an $n\in\mathbb N$ such that any closed $\lpa$-terms~$t$ validates
\begin{equation*}
\mathcal H\vdash^{\alpha+2n}_\rho\Gamma,\neg\varphi(t,\theta),\varphi(t,\psi).
\end{equation*}
\emph{Hint:} Use induction over~$\varphi$. The point is that operator forms are positive in~$X$, so that $\varphi$ has no subformula~$\neg Xs$. As $\varphi$ is an~$\lpax$-formula, we have $k(\varphi(t,\psi))\subseteq k(\psi)$. Ranks were defined for~$\lido$-formulas only, but at least intuitively, $n$ is the rank of~$\varphi$. The assumption $\omega\preceq\alpha$ ensures that clause~($\bigvee$) of Definition~\ref{def:H-controlled-proofs} can be applied when $\varphi$ begins with a quantifier over~$\mathbb N$.
\end{exercise}

To make our embedding of $\lid$ into~$\lido$ explicit, we write $\psi^+$ for the formula that results from $\psi$ when each predicate symbol $I_\varphi$ is replaced by~$I_\varphi^{\prec\Omega}$. Note that we always have $k(\psi^+)\subseteq\{\Omega\}\subseteq\mathcal H(\emptyset)$. The following result provides an embedding of axiom schema~(L). It also yields an example of a proof that has height above~$\Omega$ even though it does not use clause~($\mathsf{Fix}$). Note that a somewhat similar argument was used to embed the induction axiom in the proof of Theorem~3.7 from~\cite{first-course}.

\begin{proposition}\label{prop:deriv-L}
For any operator form~$\varphi$ and any $\lid$-formula~$\psi(x)$, we have
\begin{equation*}
\mathcal H\vdash^{\Omega\cdot 2+n}_0\forall x\in\mathbb N\,\left(\varphi(x,\psi^+)\to\psi^+(x)\right)\to\forall x\in\mathbb N\,\left(I_\varphi^{\prec\Omega}x\to\psi^+(x)\right)
\end{equation*}
for some~$n\prec\omega$ (where $\Omega\cdot 2$ denotes $\Omega+\Omega$).
\end{proposition}
\begin{proof}
Let us recall that $\cl_\varphi(\psi^+)$ abbreviates the premise of the desired implication. For $\beta:=\max\{\rk(\psi^+),1\}$ we shall use induction on~$\delta\preceq\Omega$ to show that
\begin{equation*}
\mathcal H[\delta]\vdash^{\omega\cdot\beta+\omega\cdot\delta}_0\neg\cl_\varphi(\psi^+),\neg I_\varphi^{\prec\delta} t,\psi^+(t)
\end{equation*}
holds for every closed~$\lpa$-term~$t$. The result for~$\delta=\Omega$ yields the proposition, as we find an $m\prec\omega$ with $\rk(\psi^+)\preceq\Omega+m$ and hence
\begin{equation*}
\omega\cdot\beta+\omega\cdot\Omega+n\preceq\Omega+\omega\cdot m+\Omega+n=\Omega\cdot 2+n.
\end{equation*}
For any~$\gamma\prec\delta$, the induction hypothesis and previous exercise yield a $k\in\mathbb N$ with
\begin{equation*}
\mathcal H[\gamma]\vdash^{\omega\cdot\beta+\omega\cdot\gamma+2k}_0\neg\cl_\varphi(\psi^+),\neg\varphi(t,I_\varphi^{\prec\gamma}),\varphi(t,\psi^+).
\end{equation*}
As Lemma~\ref{lem:LEM} provides $\mathcal H\vdash^{\omega\cdot\beta}_0\neg\psi^+(t),\psi^+(t)$, we can use clause~($\bigwedge$) to infer
\begin{equation*}
\mathcal H[\gamma]\vdash^{\omega\cdot\beta+\omega\cdot\gamma+2k+1}_0\neg\cl_\varphi(\psi^+),\neg\varphi(t,I_\varphi^{\prec\gamma}),\varphi(t,\psi^+)\land\neg\psi^+(t),\psi^+(t).
\end{equation*}
Now observe that we have
\begin{equation*}
\neg\cl_\varphi(\psi^+)=\exists x\in\mathbb N\left(\varphi(x,\psi^+)\land\neg\psi^+(x)\right)\simeq\textstyle\bigvee_{n\prec\omega}\varphi(n,\psi^+)\land\neg\psi^+(n).
\end{equation*}
As in Exercise~3.5 from the first lecture~\cite{first-course}, we can replace any occurrence of~$t$ by the numeral with the same value (since the notion of false literal in Definition~\ref{def:disj-conj} is unaffected). By clause~($\bigvee$) we thus obtain
\begin{equation*}
\mathcal H[\gamma]\vdash^{\omega\cdot\beta+\omega\cdot\gamma+2(k+1)}_0\neg\cl_\varphi(\psi^+),\neg\varphi(t,I_\varphi^{\prec\gamma}),\psi^+(t).
\end{equation*}
In view of $\neg I^{\prec\delta}_\varphi t\simeq\bigwedge_{\gamma\prec\delta}\neg\varphi(t,I_\varphi^{\prec\gamma})$ we can conclude by clause~($\bigwedge$). Note that we must first weaken $\mathcal H$ to $\mathcal H[\delta]$ in order to accommodate the parameter~$\delta$.
\end{proof}

Let us now derive that $\id_1$ can be embedded into our infinitary proof system.

\begin{theorem}[`Embedding']\label{thm:embedding}
Given $\id_1\vdash\psi$, we obtain $\mathcal H\vdash^{\Omega\cdot 2+n}_{\Omega+m}\psi^+$ for any nice operator $\mathcal H$ and some $m,n\in\mathbb N$.
\end{theorem}
\begin{proof}
Recall the finitary sequent calculus~$\mathsf{PL}$ for predicate logic that was presented in Section~2 of the first lecture~\cite{first-course}, where we proved that $\mathsf{PL}$ is complete without the cut rule. Given $\id_1\vdash\psi$, we have a cut free derivation
\begin{equation*}
\mathsf{PL}\vdash\neg\theta_0,\ldots,\neg\theta_{k-1},\psi
\end{equation*}
for axioms $\theta_i$ of~$\id_1$. By induction over derivations in~$\mathsf{PL}$, we find an $l\in\mathbb N$ with
\begin{equation*}
\mathcal H\vdash^{\Omega+\omega\cdot l}_0\neg\theta_0^+,\ldots,\neg\theta_{k-1}^+,\psi^+.
\end{equation*}
The crucial case in this induction concerns an axiom $\Delta,\theta,\neg\theta$ of $\mathsf{PL}$, for which we invoke Lemma~\ref{lem:LEM}. Some details on the case of a universal quantifier can be found in the proof of Theorem~3.7 from the previous course~\cite{first-course}. Concerning the case of an existential quantifier, we recall that each term in an infinitary derivation can be replaced by the corresponding numeral, as in the previous proof. Since all formula ranks have the form $\omega\cdot\alpha+j$ with $\alpha\preceq\Omega$, there is an $m\in\mathbb N$ with $\rk(\theta_i^+)\prec\Omega+m$ for all~$i<k$. To conclude by a series of cuts, it thus remains to find $n(i)$ with
\begin{equation*}
\mathcal H\vdash^{\Omega\cdot 2+n(i)}_0\theta_i\quad\text{for all }i<k.
\end{equation*}
When $\theta_i$ is an instance of axiom schema~(F) or~(L), this is accomplished by the previous propositions. The axioms of~$\mathsf{PA}\subseteq\mathsf{ID_1}$, including induction for the extended language~$\lid$, are treated as in the proof of Theorem~3.7 from~\cite{first-course}. To cover the equality axiom for~$I_\varphi$, we need derivations
\begin{equation*}
\mathcal H\vdash^{\Omega}_0 p\neq q,\neg I_\varphi^{\prec\Omega}p,I_\varphi^{\prec\Omega}q.
\end{equation*}
When $p$ and $q$ are different numerals, the true prime formula $p\neq q$ corresponds to the empty conjunction (cf.~Definition~\ref{def:disj-conj}), so that we can use clause~($\bigwedge$) without premises. In the remaining case, we can conclude by Lemma~\ref{lem:LEM}.
\end{proof}

The method that we know from the previous lecture~\cite{first-course} can be used to decrease the cut rank from $\Omega+m$ to~$\Omega+1$, as we will see in the next section. For the moment, we only recall that this method replaces a cut over~$\psi\simeq\bigvee_{\gamma\prec\delta}\psi_\gamma$ by cuts over instances $\psi_\gamma$ from which $\psi$ has been derived. Here the cut rank decreases because we have~$\rk(\psi_\gamma)\prec\rk(\psi)$. The same method cannot be used to reach cut rank~$\Omega$, because clause~($\mathsf{Fix}$) from Definition~\ref{def:H-controlled-proofs} allows to infer the disjunctive formula $I^{\prec\Omega}_\varphi t$ from the premise~$\varphi(t,I_\varphi^{\prec\Omega})$, even though we will generally have
\begin{equation*}
\rk(I^{\prec\Omega}_\varphi t)=\Omega\prec\rk(\varphi(t,I_\varphi^{\prec\Omega})).
\end{equation*}
To resolve this issue, we shall now present the aforementioned method of collapsing, which performs three different tasks: First, it transforms uncountable proof heights above~$\Omega$ into countable ones below. Secondly, this means that it must eliminate occurrences of clause~($\mathsf{Fix}$), since the condition~$\Omega\preceq\alpha$ from Definition~\ref{def:H-controlled-proofs} will no longer be satisfied. Finally, it will indeed decrease the cut rank from $\Omega+1$ to an ordinal below~$\Omega$.

One should not expect that all proofs can be collapsed to countable height. The point is that sentences of the form $\neg I_\varphi^{\prec\Omega}t\simeq\bigwedge_{\gamma\prec\Omega}\neg\varphi(t,I_\varphi^{\prec\gamma})$ are derived with \mbox{$\Omega$-}many premises, which will typically yield derivations with height above~$\Omega$ (see the proof of Proposition~\ref{prop:deriv-L} for an instructive example). To avoid this obstruction, we declare that a $\Sigma(\Omega)$-formula is an $\lido$-formula that does not have subformulas of the form~$\neg I^{\prec\Omega}_\varphi t$ (so that $I^{\prec\Omega}_\varphi$ does only occur positively). By a $\Sigma(\Omega)$-sequent we mean a sequent that does only contain $\Sigma(\Omega)$-formulas. An infinite derivation that consists of $\Sigma(\Omega)$-formulas is countably branching and hence of countable height.

We can now observe that the three tasks above are interdependent: To obtain countable proofs, we must reduce the cut rank to~$\Omega$, since a cut over $I^{\prec\Omega}_\varphi t$ would introduce a premise $\neg I^{\prec\Omega}_\varphi t$ that is no $\Sigma(\Omega)$-formula. To reduce the cut rank, we must eliminate applications of~($\mathsf{Fix}$), as explained above. To replace a derivation of~$I_\varphi^{\prec\Omega}t$ via~($\mathsf{Fix}$), we will first collapse the proof of the premise $\varphi(t,I_\varphi^{\prec\Omega})$ to some countable height~$\gamma\prec\Omega$. Then the boundedness result from Theorem~\ref{thm:boundedness} will yield a proof of $\varphi(t,I_\varphi^{\prec\gamma})$, so that $I_\varphi^{\prec\Omega}t\simeq\bigvee_{\gamma\prec\Omega}\varphi(t,I_\varphi^{\prec\gamma})$ can be re{-}derived by clause~($\bigvee$). Details can be found in the proof of Theorem~\ref{thm:collapsing} below. The following exercise provides a final ingredient, which is somewhat similar to the inversion result from the previous lecture~\cite{first-course}.

\begin{exercise}\label{ex:bounded-inversion}
Given $\mathcal H\vdash^\alpha_\rho\Gamma,\neg I^{\prec\Omega}_\varphi t$, show $\mathcal H\vdash^\alpha_\rho\Gamma,\neg I^{\prec\delta}_\varphi t$ for $\delta\prec\Omega$ with $\delta\in\mathcal H(\emptyset)$.
\end{exercise}

Finally, we present the centerpiece of our ordinal analysis (adapted from~\cite{buchholz-local-predicativity}). The reader may wish to recall the end of Section~\ref{sect:Bachmann-Howard} for some relevant notation.

\begin{theorem}[`Collapsing' or `Impredicative cut elimination']\label{thm:collapsing}
Given that $\Gamma$ is a $\Sigma(\Omega)$-sequent and that we have $\alpha\in\mathcal H_\alpha(X)$ and $X\subseteq\bigcap\{C_\xi(\vartheta\xi)\,|\,\alpha\prec\xi\}$, we get
\begin{equation*}
\mathcal H_\alpha[X]\vdash^\beta_{\Omega+1}\Gamma\qquad\Rightarrow\qquad\mathcal H_\eta[X]\vdash^{\vartheta\eta}_{\vartheta\eta}\Gamma\quad\text{for}\quad\eta=\alpha+\omega(\beta).
\end{equation*}
\end{theorem}
\begin{proof}
We argue by induction on~$\beta$. First observe that the premise of the desired implication entails
\begin{equation*}
\{\beta\}\cup k(\Gamma)\subseteq\mathcal H_\alpha(X),
\end{equation*}
by the initial condition from Definition~\ref{def:H-controlled-proofs}. In view of $\alpha\prec\eta$ we may replace~$\mathcal H_\alpha$ by $\mathcal H_\eta$, due to part~(a) of Proposition~\ref{prop:H-basic}. Let us also recall Exercise~\ref{ex:weakening}, in order to justify applications of weakening that are left implicit below. By Exercise~\ref{ex:operators} and Proposition~\ref{prop:H-basic}(b), we see that $\mathcal H_\eta(X)$ contains $\eta$ and hence~$\vartheta\eta$, which yields the initial condition that is required for $\mathcal H_\eta[X]\vdash^{\vartheta\eta}_{\vartheta\eta}\Gamma$. We now distinguish cases according to the clauses from Definition~\ref{def:H-controlled-proofs}. First consider an application of~($\mathsf{Fix}$), where $\Gamma$ contains a formula $I_\varphi^{\prec\Omega}t$ and there is a $\beta'\prec\beta$ with
\begin{equation*}
\mathcal H_\alpha[X]\vdash^{\beta'}_{\Omega+1}\Gamma,\varphi(t,I_\varphi^{\prec\Omega}).
\end{equation*}
Here $\varphi(t,I_\varphi^{\prec\Omega})$ is a $\Sigma(\Omega)$-formula, since the operator form~$\varphi$ does only contain positive occurrences of~$X$. Thus the induction hypothesis yields
\begin{equation*}
\mathcal H_{\eta'}[X]\vdash^{\vartheta\eta'}_{\vartheta\eta'}\Gamma,\varphi(t,I_\varphi^{\prec\Omega})\quad\text{with}\quad\eta'=\alpha+\omega(\beta').
\end{equation*}
By the boundedness result from Theorem~\ref{thm:boundedness}, we can infer
\begin{equation*}
\mathcal H_{\eta'}[X]\vdash^{\vartheta\eta'}_{\vartheta\eta'}\Gamma,\varphi(t,I_\varphi^{\prec\vartheta\eta'}).
\end{equation*}
Crucially, part~(c) of Proposition~\ref{prop:H-basic} ensures that we have $\vartheta\eta'\prec\vartheta\eta$ (since we get $\eta'\in\mathcal H_\alpha(X)$ as above). In view of $I_\varphi^{\prec\Omega}t\simeq\bigvee_{\gamma\prec\Omega}\varphi(t,I_\varphi^{\prec\gamma})$, we can use clause~($\bigvee$) to infer $\mathcal H_\eta[X]\vdash^{\vartheta\eta}_{\vartheta\eta}\Gamma$, as desired. Next, we consider an application of clause~($\mathsf{Cut}$), where the cut formula~$\psi$ will necessarily have rank~$\rk(\psi)\prec\Omega+1$. First assume that we have $\rk(\psi)=\Omega$, which means that $\psi$ must be of the form~$I_\varphi^{\prec\Omega}t$ or $\neg I_\varphi^{\prec\Omega}t$. For some~$\beta'\prec\beta$, our cut will then have premises
\begin{equation*}
\mathcal H_\alpha[X]\vdash^{\beta'}_{\Omega+1}\Gamma,I_\varphi^{\prec\Omega}t\quad\text{and}\quad\mathcal H_\alpha[X]\vdash^{\beta'}_{\Omega+1}\Gamma,\neg I_\varphi^{\prec\Omega}t.
\end{equation*}
Here $I_\varphi^{\prec\Omega}t$ is a $\Sigma(\Omega)$-formula while $\neg I_\varphi^{\prec\Omega}t$ is not, so that the induction hypothesis can only be applied to the first premise. Together with boundedness, it yields
\begin{equation*}
\mathcal H_{\eta'}[X]\vdash^{\vartheta\eta'}_{\vartheta\eta'}\Gamma,I^{\prec\vartheta\eta'}_\varphi t\quad\text{for}\quad\eta'=\alpha+\omega(\beta').
\end{equation*}
Concerning the second premise of our cut, we first replace~$\mathcal H_\alpha[X]$ by $\mathcal H_{\eta'}[X]$, in order to accommodate $\vartheta\eta'$ as a new parameter. Then we invoke Exercise~\ref{ex:bounded-inversion} to get
\begin{equation*}
\mathcal H_{\eta'}[X]\vdash^{\beta'}_{\Omega+1}\Gamma,\neg I_\varphi^{\prec\vartheta\eta'}t.
\end{equation*}
The point is that $\neg I_\varphi^{\prec\vartheta\eta'}t$ is a $\Sigma(\Omega)$-formula. Clearly we get $X\subseteq\bigcap\{C_\xi(\vartheta\xi)\,|\,\eta'\prec\xi\}$ due to $\alpha\prec\eta'$. We can thus use the induction hypothesis with $\eta'$ at the place of~$\alpha$. It allows us to conclude
\begin{equation*}
\mathcal H_{\eta''}[X]\vdash^{\vartheta\eta''}_{\vartheta\eta''}\Gamma,\neg I^{\prec\vartheta\eta'}_\varphi t\quad\text{with}\quad\eta''=\eta'+\omega(\beta')=\alpha+\omega(\beta')+\omega(\beta').
\end{equation*}
As in Exercise~5.4 from the first lecture~\cite{first-course}, one sees that $\omega(\beta)$ is closed under addition, so that we have $\eta''\prec\eta$. Using Proposition~\ref{prop:H-basic}, one can infer
\begin{equation*}
\rk(I^{\prec\vartheta\eta'}_\varphi t)=\omega\cdot\vartheta\eta'=\vartheta\eta'\prec\vartheta\eta''\prec\vartheta\eta.
\end{equation*}
We thus obtain $\mathcal H_\eta[X]\vdash^{\vartheta\eta}_{\vartheta\eta}\Gamma$ by an application of~($\mathsf{Cut}$) with cut formula~$I^{\prec\vartheta\eta'}_\varphi t$. Let us also consider the case where the original cut formula~$\psi$ has rank~$\rk(\psi)\prec\Omega$. Here $\psi$ and $\neg\psi$ are $\Sigma(\Omega)$-formulas (with no subformulas $I_\varphi^{\prec\Omega}t$), so that we can apply the induction hypothesis to both premises. After doing so, we can conclude by a cut over the same formula~$\psi$. To see that this yields the desired bound on the cut rank, we invoke Exercise~\ref{ex:operators} to get
\begin{equation*}
\rk(\psi)\in\mathcal H_\alpha(k(\psi))\subseteq\mathcal H_\alpha(X)\subseteq C_\eta(\vartheta\eta),
\end{equation*}
where the second inclusion relies on~$X\subseteq C_\eta(\vartheta\eta)$ and on Definition~\ref{def:H_gamma}. In view of $\rk(\psi)\prec\Omega$ we can invoke Proposition~\ref{prop:vartheta-set-def} to get $\rk(\psi)\prec\vartheta\eta$, as needed. In the penultimate case, we have an application of clause~($\bigvee$) to a formula~$\psi\simeq\bigvee_{\gamma\prec\delta}\psi_\gamma$. Here the premise provides
\begin{equation*}
\mathcal H_\alpha[X]\vdash^{\beta'}_{\Omega+1}\Gamma,\psi_\gamma
\end{equation*}
for some $\beta'\prec\beta$ and some $\gamma\prec\delta$ with $\gamma\in\mathcal H_\alpha(X)$. According to Definition~\ref{def:disj-conj}, we always have $\delta\preceq\Omega$. As in the previous case, we thus get $\gamma\prec\vartheta\eta$. This allows us to reapply clause~($\bigvee$) after the induction hypothesis has been used. Concerning the latter, we note that all~$\psi_\gamma$ are $\Sigma(\Omega)$-formulas when the same holds for~$\psi$ (also when $\psi$ is conjunctive). Finally, we consider an application of~($\bigwedge$) to $\psi\simeq\bigwedge_{\gamma\prec\delta}\psi_\gamma$, where we have premises
\begin{equation*}
\mathcal H_\alpha[X\cup\{\gamma\}]\vdash^{\beta(\gamma)}_{\Omega+1}\Gamma,\psi_\gamma\quad\text{with }\beta(\gamma)\prec\beta\text{ for all }\gamma\prec\delta.
\end{equation*}
Since $\psi\in\Gamma$ is a conjunctive $\Sigma(\Omega)$-formula, it cannot have the form $I_\varphi^{\prec\Omega}t$ or $\neg I_\varphi^{\prec\Omega}t$, so that we must have $\delta\prec\Omega$. In view of Definition~\ref{def:disj-conj} we have $\delta\preceq\omega$ or $\delta\in k(\Gamma)$, both of which yield~$\delta\in\mathcal H_\alpha(X)$. Given any~$\xi$ with $\alpha\prec\xi$, we can infer $\delta\in C_\xi(\vartheta\xi)$ and then $\delta\prec\vartheta\xi$ as above. Due to clause~(i) of Definition~\ref{def:coeff-sets}, we get
\begin{equation*}
\gamma\in\textstyle\bigcap\{C_\xi(\vartheta\xi)\,|\,\alpha\prec\xi\}\quad\text{for any }\gamma\prec\delta.
\end{equation*}
We may thus apply the induction hypothesis with $X\cup\{\gamma\}$ at the place of~$X$. Once we have done so, we can reapply clause~($\bigwedge$) to conclude.
\end{proof}

\section{Predicative cut elimination and the Veblen hierarchy}\label{sect:Veblen-hierarchy}

In this section, we shall first recover a method from the previous lecture~\cite{first-course}, which allows us to decrease the cut rank from $\rho+1$ to~$\rho$, provided the latter is different from~$\Omega$. Together with the results from above, this will suffice to obtain an independence result for the theory~$\pica^-$ of parameter-free $\Pi^1_1$-comprehension. We shall then present a method that is known as predicative cut elimination. It~uses the so-called Veblen functions to decrease the cut rank below limit ordinals other than~$\Omega$. Predicative cut elimination is central for the analysis of mathematically~significant theories of intermediate strength (see, e.\,g.,~\cite[Section~5.2]{rathjen-sieg-stanford}). In the context of stronger and impredicative theories, it plays an important but auxiliary role.

The three parts of the following exercise corresponds to the results of `inversion', `reduction' and `cut elimination' from Section~4 of the first lecture~\cite{first-course}. All ideas can be found in the latter, but the reader may wish to reflect on operator control and the role of clause~($\mathsf{Fix}$) from Definition~\ref{def:H-controlled-proofs}. A proof of the results can also be found in~\cite[Section~9.3.2]{pohlers-proof-theory}, albeit for a slightly different proof system.

\begin{exercise}\label{ex:basic-cut-elim}
(a) Given a conjunctive formula $\psi\simeq\bigwedge_{\gamma\prec\delta}\psi_\gamma$, show that $\mathcal H\vdash^\alpha_\rho\Gamma,\psi$ implies $\mathcal H\vdash^\alpha_\rho\Gamma,\psi_\gamma$ for any $\gamma\prec\delta$ with $\gamma\in\mathcal H(\emptyset)$. \emph{Hint:} The latter yields $\mathcal H[\gamma]=\mathcal H$. Crucially, the formulas $I^{\prec\Omega}_\varphi t$ in~($\mathsf{Fix}$) are disjunctive and hence different from~$\psi$.

(b) Show that $\mathcal H\vdash^\alpha_\rho\Gamma,\neg\psi$ and $\mathcal H\vdash^\beta_\rho\Gamma,\psi$ entail $\mathcal H\vdash^{\alpha+\beta}_\rho\Gamma$ when $\psi\simeq\bigvee_{\gamma\prec\delta}\psi_\gamma$ is disjunctive of rank $\rk(\psi)=\rho\neq\Omega$. \emph{Hint:} The last condition ensures that $\psi$ is different from the formulas $I^{\prec\Omega}_\varphi t$ in~($\mathsf{Fix}$). In the case of clause~($\bigvee$), Definition~\ref{def:H-controlled-proofs} provides $\gamma\in\mathcal H(\emptyset)$, as needed to invoke part~(a).

(c) Show that $\mathcal H\vdash^\alpha_{\rho+1}\Gamma$ implies $\mathcal H\vdash^{\omega(\alpha)}_\rho\Gamma$ for~$\rho\neq\Omega$.
\end{exercise}

Let us recall that each $\lid$-sentence~$\psi$ translates into an $\lido$-sentence~$\psi^+$ in which all occurrences of~$I_\varphi$ are replaced by~$I_\varphi^{\prec\Omega}$. The condition in the following result means that $\psi$ may only involve positive occurrences of the fixed point predicates~$I_\varphi$.

\begin{corollary}\label{cor:embed-plus-collapse}
If $\psi^+$ is a $\Sigma(\Omega)$-sentence, then we have
\begin{equation*}
\id_1\vdash\psi\qquad\Rightarrow\qquad\mathcal H_\eta\vdash^{\vartheta\eta}_{\vartheta\eta}\psi^+\quad\text{for some}\quad\eta\in\vartheta(\varepsilon_{\Omega+1}).
\end{equation*}
\end{corollary}
\begin{proof}
Due to the embedding result given by Theorem~\ref{thm:embedding}, we get $\mathcal H_0\vdash^\beta_{\Omega+m}\psi^+$ for some $m\in\mathbb N$ and~$\beta\in\vartheta(\varepsilon_{\Omega+1})$. By iterated applications of part~(c) of the previous exercise, we may assume $m=1$ for a modified~$\beta$. To obtain the claim for~$\eta=\omega(\beta)$, it suffices to invoke the collapsing result from Theorem~\ref{thm:collapsing}.
\end{proof}

In the previous lecture~\cite{first-course}, we have seen that Peano arithmetic cannot prove that~$\varepsilon_0$ is well founded. We now obtain an analogous result for the much stronger axiom system~$\pica^-$. The following result is sharp in a suitable sense, as one can see by adapting the argument in~\cite[Section~9.6]{pohlers-proof-theory} (see also~\cite[Section~10]{rathjen-weiermann-kruskal}). This means that the Bachmann-Howard ordinal is the proof theoretic ordinal of the axiom systems $\pica^-$ and~$\id_1$ (and also of Kripke-Platek set theory).

\begin{theorem}
The fact that the Bachmann-Howard order $\{\alpha\in\vartheta(\varepsilon_{\Omega+1})\,|\,\alpha\prec\Omega\}$ is well founded cannot be proved in the axiom system $\pica^-$.
\end{theorem}
\begin{proof}
By the previous corollary and Exercise~\ref{ex:limit-wfness-id1} we get
\begin{equation*}
\id_1\nvdash I_\varphi\Omega\quad\text{for}\quad\varphi(x,X)=\forall y\in\mathbb N\,(y\prec x\to Xy),
\end{equation*}
as any collapsing value~$\vartheta\eta$ lies below~$\Omega$. We would like to conclude by the conservativity result from Theorem~\ref{thm:pica-id-conserv}. For this purpose we recall the notation
\begin{align*}
\cl_\varphi(X)&=\forall x\in\mathbb N\,(\varphi(x,X)\to x\in X),\\
\lf_\varphi(Y)&=\forall x\in\mathbb N\,\big(x\in Y\leftrightarrow\forall X\subseteq\mathbb N\,(\cl_\varphi(X)\to x\in X)\big).
\end{align*}
Aiming at a contradiction, we assume that $\pica^-$ proves that $\prec$ is well founded on~$\vartheta(\varepsilon_{\Omega+1})\cap\Omega$. We then get $\pica^-+\lf_\varphi(Y_\varphi)\vdash\Omega\in Y_\varphi$. Indeed, given $\cl_\varphi(X)$ we obtain $\alpha\in X$ for any $\alpha\preceq\Omega$, as a minimal counterexample would satisfy $\varphi(\alpha,X)$ but $\alpha\notin X$. Now we can indeed use Theorem~\ref{thm:pica-id-conserv} to conclude $\id_1\vdash I_\varphi\Omega$.
\end{proof}

Let us now derive the independence result that was discussed in Remark~\ref{rmk:unprovability}.

\begin{corollary}\label{cor:EKT-unprovable}
The axiom system $\pica^-$ of parameter-free $\Pi^1_1$-comprehension cannot prove that $T_2(\emptyset)$ is a well partial order (i.\,e., it cannot prove Harvey Friedman's extended Kruskal theorem for trees with gap condition, even for two labels).
\end{corollary}
\begin{proof}
This follows from the previous theorem, as the statement that $T_2(\emptyset)$ is a well partial order entails that $\vartheta(\varepsilon_{\Omega+1})$ is well founded, by the proof of Corollary~\ref{cor:bachmann-howard-wo}.
\end{proof}

We continue with a set theoretic motivation of the Veblen functions, even though the latter will officially be defined in syntactic terms. Let us temporarily write~$\Omega$ for the collection of all countable ordinals. A function $f:\Omega\to\Omega$ is called normal if it is strictly increasing and continuous at limits, i.\,e., if we have
\begin{equation*}
f(\lambda)=\sup\{f(\alpha)\,|\,\alpha<\lambda\}\quad\text{for any limit }\lambda.
\end{equation*}
Equivalently, $f$ is the increasing enumeration of a subset $\rng(f)\subseteq\Omega$ that is closed an unbounded (short: club) in the following sense:
\begin{enumerate}[label=(\roman*)]
\item for each $\alpha<\Omega$ there is a $\beta\in\rng(f)$ with $\alpha<\beta$,
\item if each $\alpha<\gamma$ admits a $\beta\in\rng(f)$ with $\alpha<\beta<\gamma$, then we have $\gamma\in\rng(f)$.
\end{enumerate}
An important example of a normal function is given by
\begin{equation*}
\varphi(0,-):\Omega\to\Omega\quad\text{with}\quad\varphi(0,\beta):=\omega^\beta.
\end{equation*}
In the previous lecture~\cite{first-course} we have encountered the first fixed point~$\varepsilon_0$ of this function. More generally, the set of fixed points of any normal function is club. Given a normal function $\varphi(\alpha,-)$, we thus get another normal function
\begin{equation*}
\varphi(\alpha+1,-):=\text{``the increasing enumeration of $\{\gamma<\Omega\,|\,\varphi(\alpha,\gamma)=\gamma\}$"}.
\end{equation*}
The intersection of countably many clubs is itself club. For a limit~$\lambda<\Omega$ we put
\begin{equation*}
\varphi(\lambda,-):=\text{``the increasing enumeration of $\{\gamma<\Omega\,|\,\varphi(\alpha,\gamma)=\gamma\text{ for all }\alpha<\lambda\}$"}.
\end{equation*}
This hierarchy of normal functions $\varphi(\alpha,-):\Omega\to\Omega$ for~$\alpha<\Omega$ goes back to work of Oswald Veblen~\cite{Veblen1908} and is thus called Veblen hierarchy. It is crucial for predicative ordinal analysis and also interesting from the viewpoint of computability theory (see~\cite[Section~5.2]{rathjen-sieg-stanford} and~\cite{marcone-montalban}). By construction we have
\begin{equation*}
\varphi(\alpha,\varphi(\beta,\gamma))=\varphi(\beta,\gamma)\quad\text{for}\quad\alpha<\beta.
\end{equation*}
As $\varphi(\alpha,-)$ is strictly increasing, we get $\beta\leq\varphi(\alpha,\beta)$ by induction on~$\beta$. The following can be derived from these observations, without further use of set theory.

\begin{exercise}\label{ex:Veblen}
Show that we have
\begin{equation*}
\varphi(\alpha,\beta)<\varphi(\alpha',\beta')\quad\Leftrightarrow\quad\begin{cases}
\text{either }\alpha<\alpha'\text{ and }\beta<\varphi(\alpha',\beta'),\\
\text{or }\alpha=\alpha'\text{ and }\beta<\beta',\\
\text{or }\alpha'<\alpha\text{ and }\varphi(\alpha,\beta)<\beta'.
\end{cases}
\end{equation*}
\emph{Hint:} For both directions, start with (the negation of) the right side.
\end{exercise}

In order to recover the Veblen hierarchy in our syntactic setting, we first declare that $\Omega\cdot\alpha$ for $\alpha\in\vartheta(\varepsilon_{\Omega+1})$ is given by
\begin{equation*}
\Omega\cdot\langle\alpha_0,\ldots,\alpha_{n-1}\rangle:=\langle\Omega+\alpha_0,\ldots,\Omega+\alpha_{n-1}\rangle,
\end{equation*}
where we identify~$\vartheta\beta$ and $\Omega$ with $\langle\vartheta\beta\rangle$ and $\langle\Omega\rangle$, respectively. This is motivated by the idea that $\Omega=\omega^\Omega$ is an $\varepsilon$-number, so that the usual ordinal arithmetic yields
\begin{equation*}
\Omega\cdot\left(\omega^{\alpha_0}+\ldots+\omega^{\alpha_{n-1}}\right)=\omega^{\Omega+\alpha_0}+\ldots+\omega^{\Omega+\alpha_{n-1}}.
\end{equation*}
One readily checks that $\alpha\prec\beta$ entails $\Omega\cdot\alpha\prec\Omega\cdot\beta$. We will abbreviate~$\Omega^2:=\Omega\cdot\Omega$. The following definition does not quite give the Veblen hierarchy from above. In particular, Exercise~\ref{ex:Bachmann-Howard} yields $\beta\prec\vartheta(\Omega\cdot\alpha+\beta)$ for $\beta\prec\Omega$, so that we do not get any fixed points. Nevertheless, we will see that the crucial part of Exercise~\ref{ex:Veblen} is preserved. The precise relation with the original Veblen hierarchy is discussed~in~\cite{buchholz-bachmann-howard}.

\begin{definition}
For $\alpha,\beta\prec\Omega$ in $\vartheta(\varepsilon_{\Omega+1})$, we set $\overline\varphi(\alpha,\beta):=\vartheta(\Omega\cdot\alpha+\beta)\prec\Omega$.
\end{definition}

As promised, we recover the parts of Exercise~\ref{ex:Veblen} that will be needed below. The reader is invited to consider the remaining parts of the exercise as well.

\begin{lemma}\label{lem:Veblen}
The following holds for any $\alpha,\alpha',\beta,\beta'\prec\Omega$:

(a) We have $\overline\varphi(\alpha,\beta)\prec\overline\varphi(\alpha,\beta')$ for~$\beta\prec\beta'$.

(b) Given $\alpha\prec\alpha'$ and $\beta\prec\overline\varphi(\alpha',\beta')$, we get $\overline\varphi(\alpha,\beta)\prec\overline\varphi(\alpha',\beta')$.
\end{lemma}
\begin{proof}
(a) Due to Exercise~\ref{ex:Bachmann-Howard}(a) we obtain
\begin{equation*}
E(\alpha)\cup E(\beta')=E(\Omega\cdot\alpha+\beta')\prec^*\vartheta(\Omega\cdot\alpha+\beta').
\end{equation*}
Concerning the equality, we note that the inclusion~$\subseteq$ is readily verified for~$\beta'\prec\Omega$, while it may fail without this condition (e.\,g., since we have $\Omega\cdot\vartheta\alpha'+\Omega^2=\Omega^2$). Given $\beta\prec\beta'\prec\Omega$, we can invoke part~(e) of the cited exercise to get
\begin{equation*}
E(\Omega\cdot\alpha+\beta)\prec^*\vartheta(\Omega\cdot\alpha+\beta').
\end{equation*}
As we also have $\Omega\cdot\alpha+\beta\prec\Omega\cdot\alpha+\beta'$, we may conclude by part~(a) of the exercise.

(b) The condition $\beta\prec\overline\varphi(\alpha',\beta')$ ensures $E(\beta)\prec^*\vartheta(\Omega\cdot\alpha'+\beta')$, by Exercise~\ref{ex:Bachmann-Howard}(c). Let us also note that $\alpha\prec\alpha'$ entails $\Omega\cdot\alpha+\beta\prec\Omega\cdot\alpha'+\beta'$ for any~$\beta,\beta'\prec\Omega$. We can conclude by a similar argument as above.
\end{proof}

In our setting, the central ingredient for the ordinal analysis of predicative axiom systems can now be given as follows.

\begin{theorem}[`Predicative cut elimination']\label{thm:pred-cut-elim}
If we have $\alpha,\beta,\rho\prec\Omega$ and \mbox{$\alpha\in\mathcal H_\eta[X]$} with $\Omega^2\preceq\eta$, then we get
\begin{equation*}
\mathcal H_\eta[X]\vdash^\beta_{\rho+\omega(\alpha)}\Gamma\quad\Rightarrow\quad\mathcal H_\eta[X]\vdash^{\overline\varphi(\alpha,\beta)}_{\rho}\Gamma.
\end{equation*}
\end{theorem}
\begin{proof}
We argue by main induction on~$\alpha$ and side induction on~$\beta$. Let us first note that the assumptions entail $\alpha,\beta\in\mathcal H_\eta[X]$. In view of $E(\Omega\cdot\alpha+\beta)=E(\alpha)\cup E(\beta)$ and $\Omega\cdot\alpha+\beta\prec\Omega^2\preceq\eta$, we can invoke Exercise~\ref{ex:operators} and Proposition~\ref{prop:H-basic} to get
\begin{equation*}
\overline\varphi(\alpha,\beta)=\vartheta(\Omega\cdot\alpha+\beta)\in\mathcal H_\eta(X),
\end{equation*} 
as required by the initial condition from Definition~\ref{def:H-controlled-proofs}. We now distinguish cases according to the clauses from this definition. In the crucial case of clause~($\mathsf{Cut})$, we first use the side induction hypothesis to obtain
\begin{equation*}
\mathcal H_\eta[X]\vdash^{\overline\varphi(\alpha,\beta')}_{\rho}\Gamma,\psi\quad\text{and}\quad\mathcal H_\eta[X]\vdash^{\overline\varphi(\alpha,\beta')}_{\rho}\Gamma,\neg\psi
\end{equation*}
with~$\beta'\prec\beta$ and $\rk(\psi)\prec\rho+\omega(\alpha)$. If we have $\rk(\psi)\preceq\rho$, then we can conclude by a cut or the reduction result from Exercise~\ref{ex:basic-cut-elim}(b), i.\,e., by the basic cut elimination method from the previous lecture~\cite{first-course}. Otherwise, we may write
\begin{equation*}
\rk(\psi)=\rho+\langle\alpha_0,\ldots,\alpha_n\rangle=\rho+\omega(\alpha_0)+\ldots+\omega(\alpha_n)\quad\text{with}\quad\alpha_n\preceq\ldots\preceq\alpha_0\prec\alpha.
\end{equation*}
Let us observe that $\rk(\psi)\in\mathcal H_\eta(k(\psi))\subseteq\mathcal H_\eta(X)$ holds by Exercise~\ref{ex:operators} and the initial condition from Definition~\ref{def:H-controlled-proofs}. In view of $E(\alpha_i)\subseteq E(\rk(\psi))$ we get $\alpha_i\in\mathcal H_\eta(X)$. By an application of~($\mathsf{Cut}$), we derive
\begin{equation*}
\mathcal H_\eta[X]\vdash^{\overline\varphi(\alpha,\beta')+1}_{\rho'}\Gamma\quad\text{for}\quad\rho'=\rk(\psi)+1=\rho+\omega(\alpha_0)+\ldots+\omega(\alpha_n)+\omega(0).
\end{equation*}
Iterated applications of the main induction hypothesis yield
\begin{equation*}
\mathcal H_\eta[X]\vdash^{\beta_{n+1}}_\rho\Gamma\quad\text{with}\quad\beta_0:=\overline\varphi(0,\overline\varphi(\alpha,\beta')+1)\text{ and }\beta_{i+1}:=\overline\varphi(\alpha_{n-i},\beta_i).
\end{equation*}
We get $\overline\varphi(\alpha,\beta')+1\prec\overline\varphi(\alpha,\beta)$ by part~(a) of Lemma~\ref{lem:Veblen}. Using part~(b) of the latter, we inductively obtain $\beta_i\prec\overline\varphi(\alpha,\beta)$, so that we can conclude by weakening. In the remaining cases of clause~($\bigwedge$) and ($\bigvee$) from Definition~\ref{def:H-controlled-proofs}, it is straightforward to reduce to the induction hypothesis.
\end{proof}

Let us note that predicative cut elimination and collapsing become intertwined in the ordinal analysis of impredicative axiom systems beyond~$\id_1$. For these, a previous collapsing step may produce a cut rank of the form $\Omega+\rho$ with~$\omega\preceq\rho$. By a variant of predicative cut elimination, one can then reach cut rank $\Omega+1$, which allows a next collapsing step akin to Theorem~\ref{thm:collapsing}. To see a precise argument that involves these ideas, one may wish to consider the proof of Theorem~4.8 in~\cite{buchholz-local-predicativity}. Here we use predicative cut elimination to deduce the following strengthening of Corollary~\ref{cor:embed-plus-collapse}, which is relevant for more refined independence results and formalized consistency proofs (see in particular Theorem~5.16 of~\cite{simpson85}).

\begin{corollary}
If $\psi^+$ is a $\Sigma(\Omega)$-sentence, then we get
\begin{equation*}
\id_1\vdash\psi\qquad\Rightarrow\qquad\mathcal H_\alpha\vdash^\beta_0\psi^+\quad\text{for some }\alpha,\beta\in\vartheta(\varepsilon_{\Omega+1})\text{ with }\beta\prec\Omega.
\end{equation*}
\end{corollary}
\begin{proof}
By Corollary~\ref{cor:embed-plus-collapse} we get $\mathcal H_\eta\vdash^{\vartheta\eta}_{\vartheta\eta}\psi^+$ for some~$\eta$. Weakening allows us to replace~$\mathcal H_\eta$ by some~$\mathcal H_\alpha$ with $\Omega^2,\eta\preceq\alpha$. To obtain the claim for~$\beta:=\overline\varphi(\vartheta\eta,\vartheta\eta)$, it suffices to invoke the predicative cut elimination result from Theorem~\ref{thm:pred-cut-elim}.
\end{proof}

We hope that the present and previous course~\cite{first-course} have conveyed some of the fascinating ideas and applications of ordinal analysis. Once again, we stress that we have focused on a few selected topics in order to give full technical details. To complement these by a broader picture of the field, we strongly recommend to look at the survey papers by Michael Rathjen~\cite{rathjen-realm} and by Rathjen and Wilfried Sieg~\cite{rathjen-sieg-stanford}.

\bibliographystyle{amsplain}
\bibliography{Lecture_Adv-Ord-Ana}

\end{document}